\documentclass[11pt,letterpaper]{amsart}

\usepackage{amssymb}
\usepackage{amsmath}
\usepackage{amsthm}
\usepackage{enumerate}
\usepackage{tikz}
\usepackage{wrapfig, framed, caption}
\usepackage{float}
\usetikzlibrary{arrows}
\usepackage[font=small,labelfont=bf]{caption}
\usepackage{xcolor}
\usepackage{mathtools}
\usepackage[colorlinks=true, linkcolor=blue, citecolor=blue,
pagebackref=true]{hyperref}
\usepackage{fouriernc}
\usepackage{ulem}
\usepackage{setspace}

\newtheorem{theorem}{Theorem}[section]
\newtheorem*{theorem*}{Theorem}
\newtheorem{corollary}{Corollary}[section]
\newtheorem*{corollary*}{Corollary}

\newtheorem{lemma}{Lemma}[section]
\newtheorem*{lemma*}{Lemma}

\newtheorem*{claim*}{Claim}

\theoremstyle{definition}
\newtheorem{definition}{Definition}
\theoremstyle{remark}
\newtheorem{remark}{Remark}
\newtheorem*{remark*}{Remark}

\newtheorem{ques}{Question}

\renewcommand{\Bbb}[1]{\mathbb{#1}}
\newcommand{\N}{{\Bbb N}}         % natural numbers

\newcommand{\Q}{{\Bbb Q}}         % rational numbers
\newcommand{\R}{{\Bbb R}}        % real numbers
\newcommand{\Z}{{\Bbb Z}}         % integer numbers

\newcommand{\cH}{{\mathcal H}}
\newcommand{\cI}{{\mathcal I}}

\newcommand{\cK}{{\mathcal K}}
\newcommand{\cL}{{\mathcal L}}

\newcommand{\cP}{{\mathcal P}}
\newcommand{\cQ}{{\mathcal Q}}

\newcommand{\bp}{\mathbf{p}}

\newcommand{\br}{\mathbf{r}}
\newcommand{\bv}{\mathbf{v}}

\newcommand{\bs}{\mathbf{s}}
\newcommand{\bn}{\mathbf{n}}

\newcommand{\Bad}{\mathbf{Bad}}

\DeclareMathOperator{\dimh}{\dim_H}
\DeclareMathOperator{\dimb}{\dim_B}
\DeclareMathOperator{\dimub}{\overline{\dim}_B}

\DeclareMathOperator{\dimlb}{\underline{\dim}_{B}}

\DeclareMathOperator{\dimp}{\dim_P}

\newcommand{\bx}{\mathbf{x}}
\newcommand{\by}{\mathbf{y}}
\newcommand{\bz}{\mathbf{z}}

\newcommand{\bu}{\mathbf{u}}

\newcommand{\bX}{\mathbf{X}}
\newcommand{\bq}{\mathbf{q}}

\title[The $d$-dimensional Lagrange spectrum]{Distribution of points near the origin in the $d$-dimensional Lagrange spectrum}

\author[Benjamin Ward]{ Benjamin Ward}
\address{Benjamin Ward,  {Department of Mathematics, University of York, Deramore Lane, York, YO10 5GH, UK.} }
\email{benjamin.ward@york.ac.uk, ward.ben1994@gmail.com }
\begin{document}

    \begin{abstract}
        We develop a new framework, inspired by Schmidt's games, to study the Lagrange spectrum for simultaneous Diophantine approximation in dimension $d\geq 2$. We show the Hausdorff dimension of the set of points in $\mathbb{R}^{d}$ whose best approximation constant lies in $[\varepsilon,\varepsilon(1+\delta\varepsilon^{d})]$, for some constant $\delta>0$, is positive, and for a slightly larger set approaches full dimension as $\varepsilon \to 0$. The proof combines a novel application of the Simplex lemma near rational points with the game-theoretic framework. We also use an elementary observation to show that the naturally defined Lagrange spectrum for systems of linear forms is uncountable in the case of square matrices.
    \end{abstract}
\maketitle
%%%%%%%%%%%%%%%%%%%%%%%%%%%%%%%%%
%
%
%   INTRO AND BACKGROUND
%
%
%%%%%%%%%%%%%%%%%%%%%%%%%%%%%%%%%
\section{Introduction}
Fix $d\in \N$ and let $\|\cdot\|$ be an arbitrary norm on $\R^{d}$. It is a well-known result in Diophantine approximation that there exists constant $D_{d,\|\cdot\|}>0$ depending on $d$ and $\|\cdot\|$ such that for every $\bx\in\R^{d}$
\begin{equation*}
    \left\| \bx -\tfrac{\bp}{q}\right\|< D_{d,\|\cdot\|}q^{-1-\tfrac{1}{d}}
\end{equation*}
has infinitely many solutions $(\bp,q)\in \Z^{d}\times \N$. That is to say, for every $\bx$ there exists constant
\begin{equation} \label{eq: lagrange bounds}
    0 \leq \Theta_{d}(\bx, \|\cdot\|):=\liminf_{q\to\infty}\left(q^{\tfrac{1}{d}}\min_{\bp\in\Z^{d}}\|q\bx-\bp\|\right)\leq D_{d, \|\cdot\|}\, ,
\end{equation}
which is often called the \textit{best approximation constant of $\bx$}. A substantially more difficult problem is to determine the optimal constant $D_{opt}(d,\|\cdot\|)$ for which \eqref{eq: lagrange bounds} holds. In dimension $d=1$ one can take $D_{opt}(1)=\tfrac{1}{\sqrt{5}}$ as shown by Hurwitz. In higher dimensions, the optimal value is unknown except for the case $d=2$ and $\|\cdot\|_{2}$ the Euclidean norm, where it was proven that $D_{opt}(2,\|\cdot\|_{2})=\sqrt{\frac{2}{\sqrt{23}}}$ \cite{DavenportMahler46}. See \cite{Nowak16} for a recent survey and known bounds on $D_{opt}(d,\|\cdot\|)$. 
In the other direction, Khintchine's Theorem \cite{Khintchine1949} implies that for Lebesgue almost every $\bx\in \R^{d}$ we have $\Theta_{d}(\bx,\|\cdot\|)=0$. 

Studying the set of possible values of $\Theta_{d}(\bx,\|\cdot\|)$ taken over $\bx\in \R^{d}$ defines the \textit{$d$-dimensional Lagrange Spectrum with respect to $\|\cdot\|$},
\begin{equation*}
    \cL_{d}(\|\cdot\|):=\left\{\Theta_{d}(\bx,\|\cdot\|):  \bx \in \R^{d}\right\}\footnote{Note that in the one-dimensional case authors often define the Lagrange spectrum as $\left\{ \Theta_{1}(x)^{-1}: x\in [0,1) \right\}$. We follow the notation as presented in \cite{AkhundzhanovMoshchevitin2006,Akhunzhanov2013}.}
    \subseteq \left[0,D_{d,\|\cdot\|}\right]
\end{equation*}
as introduced in \cite{AkhundzhanovMoshchevitin2006}. For each $\varepsilon >0$ we consider the corresponding level set
\begin{equation*}
    L_{d}(\varepsilon,\|\cdot\|):=\left\{ \bx \in \R^{d}: \Theta_{d}(\bx, \|\cdot\|)=\varepsilon \right\}.
\end{equation*}
Clearly $L_{d}(\varepsilon,\|\cdot\|)\neq \emptyset$ if and only if $\varepsilon \in \cL_{d}(\|\cdot\|)$.

The one-dimensional Lagrange spectrum, $\cL_{1}$, is remarkably well understood. The following structural properties are known:
\begin{itemize}
    \item $\mathcal{L}_{1}\cap \left(\tfrac{1}{3},1\right]$ is a countable discrete collection of points accumulating at $\frac{1}{3}$, \cite{Markoff}. Namely,
    \begin{equation*}
        \left\{ \tfrac{1}{\sqrt{5}}>\tfrac{1}{\sqrt{8}}>\dots > \tfrac{1}{\sqrt{9-\tfrac{4}{m_{n}^{2}}}} >\dots \, \,   :\,\,   m_{n} \, \text{ the $n$th Markoff number }\right\}.
    \end{equation*}
    Moreover, for every $\varepsilon \in \mathcal{L}_{1}\cap \left(\tfrac{1}{3},1\right]$ the set $L_{1}(\varepsilon)$ is countable.
    \item $\mathcal{L}_{1}\cap [0, c_{F}]$ is an interval called Hall's ray \cite{Hall1947}. The optimality of the constant
    \begin{equation*}
        c_{F}=\tfrac{491993569}{2221564096+283798\sqrt{462}}\approx 0.220856\ldots
    \end{equation*}
    was calculated by Freimann in \cite{Freiman1975}.

    \item $\mathcal{L}_{1}\cap \left(\tfrac{1}{3}-\varepsilon,\tfrac{1}{3}\right)$ has positive Box and Hausdorff dimension for any $\varepsilon>0$, as shown by Moreira \cite{Moreira2018}, see also \cite{EraGutMorRom2024}. Moreover, it was shown that $\cL_{1}\cap \left( c_{F}, \tfrac{1}{\sqrt{12}}\right)$ has full Hausdorff dimension.
    \item Of particular relevance to the present paper, it was shown in \cite[Theorem 2]{Moreira2018} that the level sets tend to full dimension. That is,
    \begin{equation*}
        \lim_{\varepsilon\to 0} \dimh L_{1}(\varepsilon)=1.
    \end{equation*}
    \item For any $\varepsilon>0$ in the interior of $\cL_{1}$ the dimension of $L_{1}(\varepsilon)$ is strictly increasing \cite{MoreriaVillamil25}.
\end{itemize}
For many other results and properties on the one-dimensional Lagrange spectrum see the surveys \cite{CusickFlahive, Matheus2018}. In all of the one dimensional results mentioned above the theory of continued fractions is essential, see for example \cite{Khintchine1949} for the definition and range of properties they possess. In particular, Perron's formula (see \cite{Perron21} or \cite[Section 1.4]{Matheus2018}) gives an equivalent definition of the Lagrange spectrum in terms of the continued fraction expansion of a point. The formula motivates a dynamical perspective on the Lagrange spectrum which can be generalised to describe a Lagrange spectrum for Diophantine approximation by Fuchsian groups as initiated in \cite{Lehner1952}. See for example \cite{Patterson1976, MLChallray2020, Moreira2023}
%\cite{HaasSeries1986,saar2002,Kim2022}
for the Markoff-Lagrange spectrum within this framework.\par

\subsection{The $d$-dimensional Lagrange spectrum}
Very little is known about the higher dimensional Lagrange spectrum $\cL_{d}(\|\cdot\|)$. For instance, as mentioned earlier, in the majority of cases we do not know the rightmost endpoint of $\cL_{d}(\|\cdot\|)$. The following questions are challenging open problems:
\begin{ques}\footnote{This question should be attributed to Nikolay Moshchevitin. Following an `Online Seminar in Diophantine Approximation and Related Topics' in 2021 Nikolay introduced me to the papers \cite{AkhundzhanovMoshchevitin2006, Akhunzhanov2013} and asked the above question.}
Is the set $\cL_{d}(\|\cdot\|)$ uncountable? Moreover, does it contain an interval or have positive Hausdorff dimension?
\end{ques}

\begin{ques}
    Given $\varepsilon\in \cL_{d}(\|\cdot\|)$, what can be said about $\dimh L_{d}(\varepsilon, \|\cdot\|)$?
\end{ques}
In \cite{Akhunzhanov2013}, see also \cite{AkhundzhanovMoshchevitin2006}, it was proven that for constant $C>0$ sufficiently small and $\delta>0$, for every $C>\varepsilon>0$ the closed interval 
\begin{equation*}
    \cL_{d}(\|\cdot\|) \cap [\varepsilon,\varepsilon(1+\delta\varepsilon^{1+1/d})]    
\end{equation*}
is non-empty. It was also shown that
\begin{equation*}
    \left\{ \bx \in \R^{d}: \Theta_{d}(\bx, \|\cdot\|)\in [\varepsilon,\varepsilon(1+\delta\varepsilon^{1+1/d})] \right\}
\end{equation*}
is uncountable.

In a very recent breakthrough, Kleinbock \cite{Kleinbock} improved upon the former statement, showing that $$\cL_{d}(\|\cdot\|) \quad \text{is dense on the interval} \quad \left[0,\sup_{\bx \in \R^{d}}\Theta_{d}(\bx, \|\cdot\|)\right].$$ Moreover, it was shown that for any interval $[a,b]\subset[0,\sup_{\bx \in \R^{d}}\Theta_{d}(\bx, \|\cdot\|)]$, the set
\begin{equation*}
        \left\{\bx \in \R^{d}: \Theta_{d}(\bx, \|\cdot\|)\in [a,b]\right\}
\end{equation*}
is dense in $\R^{d}$. See \cite[Theorem 1.1]{Kleinbock}.

Our main theorem gives quantitative information of a different nature: it shows that arbitrarily thin windows near the origin are realised by sets of positive Hausdorff dimension.

\begin{theorem}\label{main}
    Take $C>0$ sufficiently small and $\|\cdot\|_{2}$ the Euclidean norm. There exists $\delta>0$ such that for every $C>\varepsilon>0$ the closed interval 
    \begin{equation*}
    [\varepsilon,\varepsilon(1+\delta\varepsilon^{d})]\cap \cL_{d}(\|\cdot\|_{2})    
    \end{equation*}
    is non-empty. 
    % In particular, for $d\geq 2$ we have
    % \begin{equation} \label{metric statement}
    %     \dimlb \cL_{d}(\|\cdot\|_{2}) \geq 1-\tfrac{1}{d+1}\, ,
    % \end{equation}
    % where $\dimlb$ denotes the lower box dimension. 
    Furthermore, 
    \begin{equation*}
        \dimh \left\{ \bx \in \R^{d}: \varepsilon \leq \Theta_{d}(\bx, \|\cdot\|_{2}) \leq \varepsilon(1+\delta \varepsilon^{d})\right\}>0,
    \end{equation*}
    and for any function $\Delta:\R_{+}\to\R_{+}$, with $\Delta(r)\to \infty$ as $r\to 0$ we have that
    \begin{equation*}
    \lim_{\varepsilon\to0}\, \dimh \left\{ \bx \in \R^{d}: \varepsilon \leq \Theta_{d}(\bx, \|\cdot\|_{2}) \leq \varepsilon(1+\Delta(\varepsilon) \varepsilon^{d})\right\}=d\, . 
\end{equation*}

\end{theorem}
\begin{remark} \rm
    The recent preprint of Kleinbock \cite{Kleinbock} supersedes the non-empty statement of Theorem~\ref{main}, but not the Hausdorff dimension statements. 
    %In particular, neither of the statements of Corollary~\ref{cor: levelsets or uncountable} can be deduced from \cite{Kleinbock}. 
    In terms of \cite{Akhunzhanov2013}, their constants $C,\delta>0$ appear to be better than ours, but our windows are much thinner for small $\varepsilon>0$. For explicit value of $C, \delta$ see \eqref{eq: delta bound}, \eqref{eq: epsilon bound}. All three results use significantly different methods of proof.
\end{remark}
\begin{remark}
    Note the Hausdorff dimension of the set $$\left\{ \bx \in \R^{d}: \varepsilon \leq \Theta_{d}(\bx, \|\cdot\|_{2}) \leq \varepsilon(1+\Delta(\varepsilon) \varepsilon^{d})\right\}$$ is always strictly less than $d$, and so our asymptotic is optimal. This follows from considering the closely related set of \textit{$\varepsilon$-Badly approximable points}, $$\Bad_{d}(\varepsilon, \|\cdot\|):=\left\{ \bx\in [0,1]^{d}: \Theta_{d}(\bx, \|\cdot\|)\geq\varepsilon\right\}.$$ For bounds on the Hausdorff dimension of these sets see \cite{Jarnik1928, Kurzweil1951, Hensley1992, SWeil2015, BroderickKleinbock2015, Simmons2018, DasFishmanSimmonsUrbanski2023}. For $d\geq2$ the best known bound is due to Simmons \cite{Simmons2018}, who calculated the asymptotic
\begin{equation*}
    \dimh \Bad_{d}(\varepsilon)=d- V_{d,\|\cdot\|} \varepsilon^{d} + o(\varepsilon^{d})\, ,
\end{equation*}
where $V_{d,\|\cdot\|}$ is a constant that can be explicitly calculated.
\end{remark}
\begin{remark}
Throughout the proof we used the fact that our balls were constructed using the Euclidean norm, particularly in \S~\ref{sec: target surface} and the proof of Lemma~\ref{lem: spyglass lem}. It seems reasonable that Theorem~\ref{main} should hold for other norms (with modified parameter bounds). We should note that the results in \cite{Akhunzhanov2013, Kleinbock} hold for arbitrary norms.
\end{remark}

The proof of Theorem~\ref{main} has two main ingredients. Firstly, we give a game-theoretic framework adapted to sets of exact approximation order. Classical Schmidt and absolute winning games are suited for $\varepsilon$-badly approximable sets, but the sets in Theorem~\ref{main} require simultaneously imposing an upper bound on their best approximation constant. To accommodate these competing requirements, we introduce a rapid absolute winning game in which Alice follows a neighbourhood-avoidance strategy at most turns, but at a sequence of exceptional turns is allowed to force the play into a prescribed smaller ball. Broadly, the idea behind the game is motivated by proof techniques appearing in \cite{Bu03, BugMor2011, Moreira2018, KLWZ24, BandiDeSaxce, BakerWard} when studying sets of exact approximation order. In a recent preprint \cite{HusSim} Hussain and Simmons developed a generalisation of the \text{Rapid winning game} of \cite{HatefiSimmons2024}. In \cite[\S 6 (3)]{HusSim} they ask for a higher dimensional analogue of their game; the game we present here can be viewed as such analogue. We establish Hausdorff and packing-dimension estimates for winning sets, see Theorem~\ref{winning dimension}. We expect the framework given here to have applications beyond Theorem~\ref{main}. 

The second part is to show that the sets appearing in Theorem~\ref{main} are winning for our game. The strategy employed for standard turns in the game are essentially identical to that in \cite{BroderickKleinbock2015} for the hyperplane absolute winning game, see \cite{BroderickFishmanKleinbockReichWeiss2012}. For the exceptional turns we begin by finding a rational point $\frac{p}{q}$ that in some sense has a locally small denominator, this is obtained by Lemma~\ref{lem:eventually see a rational}. We then target a surface $S$, that can be thought of as $\partial B\left(\tfrac{p}{q}, \varepsilon q^{-1-1/d}\right)$, and descend to the surface via a novel application of the Simplex Lemma \cite[p.57]{Schmidtbook1980}, while simultaneously avoiding rational points. This is achieved in Lemma~\ref{lem: spyglass lem}.

We finish the paper by considering the Lagrange spectrum for matrices. We use an elementary observation to show that the set of best approximation constants for matrices $\bX \in \R^{n\times n}$ is uncountable and contains a Hall's Ray, answering the analogue of Question 1 in this setting. See Theorem~\ref{thm: matrices case}.

% The article is laid out as follows: We begin by describing a new version of Schmidt games that will be used to prove Theorem~\ref{main}. In particular we provide lower bounds on the Hausdorff and packing dimension of our winning sets, see Theorem~\ref{winning dimension}. This is proven in \S~\ref{winning dim proof}. We expect the framework to have applications beyond the remit of this paper. In \S~\ref{sec: proof of main} we go on to the main part of the paper which is to show that the sets given in Theorem~\ref{main} are winning within our framework. We finish with some further remarks on the Lagrange spectrum for matrices. In particular, we show that the set of best approximation constants for matrices $\bX \in \R^{n\times n}$ is uncountable and contains a Hall's Ray, thus answering the analogue of Question 1 in this setting. See Theorem~\ref{thm: matrices case} for details. 

\textbf{Acknowledgements:} The author would like to thank Nikolay Moshchevitin for introducing them to the problem studied in this paper. They would also like to thank the organisers of the Simons Semester conference on Continued Fractions, Fractals, Ergodic theory and Dynamics (Simons Foundation grant award no. SFI-MPS-T-Institutes-00010825) where this was presented, and Harold S. Erazo for pointing out \cite{MoreriaVillamil25}. The author is a Leverhulme Trust Early Career Research Fellow, ECF-2024-401.

%%%%%%%%%%%%%%%%%%%%%%%%%%%%%%%%%
%
%
%   SCHMIDT GAMES
%
%
%%%%%%%%%%%%%%%%%%%%%%%%%%%%%%%%%

\section{A general framework for the difference of two limsup sets} \label{section:Schmidt games}

The introduction of Schmidt's $(\alpha,\beta)$-games \cite{Schmidt1966} to metric number theory has resulted in the development of many groundbreaking results, see for example \cite{KleinbockWeiss2010,BroderickFishmanKleinbockReichWeiss2012,McMullen2010,BerNesYan2022}. The sets studied in this paper are not winning in the classic sense since they do not have full Hausdorff dimension. To that end we introduce a new game for which our sets are in some sense winning. For a survey of different games and the properties that winning sets (with respect to the game being played) possess see \cite{BadziahinHarrapNesharimSimmons2018}.\par 
\subsubsection{The playground}
Prior to introducing our game we provide a list of requirements needed. Note when it comes to the application of our game we will be exclusively thinking of the metric space $\R^{d}$ with norm $\|\cdot\|$ and equipped with the $d$-dimensional Lebesgue measure, but throughout this section we try to be as general as possible.\par 
Given a complete metric space $(X,d)$ equipped with a measure $\mu$ we ask for the following:
\begin{itemize}
    %\item[(COM)] $(X,d)$ is a \textit{compact} metric space.
    \item[(AHL)] $\mu$ is an \textit{Ahlfors $\delta$-regular} measure on $X$. That is, for $r_{0}>0$ sufficiently small there exist constants $0<a_{\mu}\leq1\leq b_{\mu}<\infty$ such that for all $0<r\leq r_{0}$ and $x\in X$
    \begin{equation*}
        a_{\mu}r^{\delta}\leq \mu(B(x,r))\leq b_{\mu}r^{\delta}\, .
    \end{equation*}
    \item[(DEC)] Let $\cP\subset \P(X)$ be a collection of subsets of $X$ and let $\eta, C_{dec},r_{0}>0$. A measure $\mu$ is \textit{absolutely $(\eta,\cP)$-decaying} if for every ball $B(x,r)$ with $x \in X$ and $r_{0}>r>0$, for every $P\in\cP$ and every $\varepsilon>0$ we have
\begin{equation*}
    \mu\left( B(x,r) \cap \Delta(P,\varepsilon r) \right) \leq C_{dec}\varepsilon^{\eta}\mu(B(x,r))\, ,
\end{equation*}
where for set $S\subset X$ and $r>0$ 
\begin{equation*}
    \Delta(S,r):=\left\{ x\in X: \inf_{s\in S} d(s,x) \leq r \right\}\, .
\end{equation*}
\end{itemize}
The notion of an absolutely $(\eta,\cP)$-decaying measure is well-studied, see for instance \cite{KLW04, KW05}. For example, the set of affine hyperplanes in $\R^{d}$, $\cH$, with $d$-dimensional Lebesgue measure $\lambda_{d}$ is absolutely $(1,\cH)$-decaying.\par
\subsubsection{The $(N_{i},R_{i})$-rapid $(\beta,\rho,\cP)$-absolute winning game} We introduce the following game:

\begin{definition}%$(N_{i},R_{i})$-rapid $(\beta,\rho,\cP)$-absolute winning game]
    Let $\beta,\rho>0$ and let $\cP$ be a collection of subsets of $X$. Fix a sequence of integer pairs $(N_{i},R_{i})\in \N^{2}$ with $N_{i}$ strictly increasing as $i\to\infty$. The game is as follows:
    \begin{itemize}
        \item Bob starts by playing any ball $B_{0}:=B(x_{0},\rho_{0})$ with $\rho_{0}\geq \rho$. Alice can respond by picking any $P_{0}\in \cP$ and removing the $\beta\rho_{0}$-neighbourhood. That is, Alice removes some $A_{0}=\Delta(P_{0},\beta\rho_{0})$.
        \item On Bobs second turn he can play any ball $B_{1}=B(x_{1},\rho_{1})$ such that
        \begin{equation*}
            B_{1}\subseteq B_{0}\setminus A_{0}\, , \quad \text{ and } \quad \rho_{1}= \beta \rho_{0}\, .
        \end{equation*}
        Alice can respond by picking a subset $P_{1}\in \cP$ and $A_{1}=\Delta(P_{1},\beta\rho_{1})$ which Bob cannot intersect on his next turn.
        \item The game continues in this way up to turn $N_{1}+1$. At this point Bob has just played ball $B_{N_{1}}=B(x_{N_{1}},\rho_{N_{1}})$ such that
        \begin{equation*}
            B_{N_{1}}\subset B_{N_{1}-1}\setminus A_{N_{1}-1}\, , \quad \text{ and } \quad \rho_{N_{1}}=\beta\rho_{N_{1}-1}\, .
        \end{equation*}
        On this turn Alice can respond by picking any ball $$A_{N_{1}}=B(x_{N_{1}+1},\beta^{R_{1}}\rho_{N_{1}}) \subset B_{N_{1}}\, .$$ This ball automatically becomes Bobs $(N_{1}+1)$th turn. That is, $B_{N_{1}+1}=A_{N_{1}}$. 
        \item The game then continues as usual. That is, suppose Bob has just played ball $B_{n}=B(x_{n},\rho_{n})$. If $n\not\in (N_{i})_{i\in \N}$ then Alice can respond by choosing some $P_{n}\in \cP$ and playing $A_{n}= \Delta(P_{n}, \beta\rho_{n})$ which Bob must avoid in his next turn. If $n\in (N_{i})_{i\in\N}$, say $n=N_{j}$, then Alice can respond by playing a ball $A_{n}=B(x_{n+1},\beta^{R_{j}}\rho_{n})$, which automatically becomes Bobs $(n+1)$th ball.
    \end{itemize}
    The game continues indefinitely. If Bob cannot play a legal ball we say that Alice wins by default. Otherwise, let $x_{\infty}:=\lim_{n\to\infty} B_{n}$ be the limit point of Bobs sequence of legal moves. Observe that Bobs sequence of turns are nested and the radii are decreasing, so the limit point $x_{\infty}$ exists. Call $x_{\infty}$ the outcome of the game. Let $S\subset X$. We say $S$ is a $(N_{i},R_{i})$-rapid $(\beta,\rho,\cP)$-absolute winning set if Alice has a strategy guaranteeing that $x_{\infty} \in S$.
\end{definition}
We prove the following bounds on the Packing and Hausdorff dimension of winning sets with respect to our framework.
\begin{theorem} \label{winning dimension}
Let $(X,d)$ be a complete metric space equipped with a $\delta$-Ahlfors regular measure $\mu$. Let $\cP$ be a collection of closed subsets of $X$ and suppose that $\mu$ is absolutely $(\eta,\cP)$-decaying. Suppose $S\subset X$ is $(N_{i},R_{i})$-rapid $(\beta,\rho,\cP)$-absolutely winning for some 
\begin{equation} \label{eq: geometric beta bound}
    0<\beta<\left(C_{dec}^{-1}\tfrac{a_{\mu}^{2}}{b_{\mu}^{2}}2^{-2\delta-\eta-2}\right)^{\tfrac{1}{\eta}} .
\end{equation}
and $r_{0}>\rho>0$, for $r_0>0$ as in (AHL) and (DEC). Then, for every ball $B\subset X$ with $r_{0}>r(B)\geq \rho$ we have that
\begin{align*}
    \dimh (S\cap B) \geq 
    \liminf_{n\to\infty} \left(1-\frac{\sum_{i=1}^{n}R_{i}}{N_{n}+1+\sum_{i=1}^{n}(R_{i}-1)}\right)
    \left(
    \delta-\frac{\log (c_{1}-c_{2}\beta^{\eta})}{\log \beta }
    \right)\, , \\
    \dimp (S\cap B) \geq \limsup_{n\to \infty} \left(1- \frac{\sum_{i=1}^{n-1}R_{i}}{N_{n}+\sum_{i=1}^{n-1}(R_{i}-1)} \right)\left(
    \delta-\frac{\log (c_{1}-c_{2}\beta^{\eta})}{\log \beta } \right)
\end{align*}
where $c_{1}=\tfrac{1}{2}\frac{a_{\mu}}{b_{\mu}}\left(\frac{1}{4}\right)^{\delta}$ and $c_{2}=C_{dec}\frac{b_{\mu}}{a_{\mu}}2^{\eta+1}$.

\end{theorem}

\begin{remark} \rm
    If we can choose the pair of sequences $(N_{i},R_{i})$ for every $\beta>0$ such that
    \begin{equation} \label{eq: limit to zer}
        \limsup_{n\to\infty} \frac{\sum_{i=1}^{n}R_{i}}{N_{n}+1+\sum_{i=1}^{n}(R_{i}-1)}=0 ,
    \end{equation}
    then Theorem~\ref{winning dimension} implies $\dimh (S\cap B) \geq \delta$.
\end{remark}

\begin{remark}
A few remarks on our game with respect to other previously known games are in order. As in Schmidt's original game, Bob's ordinary moves are required to have radius exactly $\beta$ times the previous radius, unlike McMullen's Strong winning game \cite{McMullen2010}. Without the sequence of integer pairs $(N_{i},R_{i})$ the game is similar to the absolute winning game, see for example \cite{McMullen2010, BFS2019}. The addition of the subsequence of turns where Alice has more control over the outcome of the game is similar to the Rapid winning and $\Psi$-Rapid winning game \cite{HatefiSimmons2024,HusSim}.
\end{remark}

% \begin{remark}
% The motivation behind the definition of this game is simple. In order to belong to the sets of interest in this paper one is required to belong to $\Bad_{d}(\varepsilon, \|\cdot\|)$ while simultaneously avoiding the set $\Bad_{d}(\varepsilon(1+\delta\varepsilon^{d}), \|\cdot\|)$. To do this the game above is designed so that one can in some sense take points belonging to $\Bad_{d}(\varepsilon, \|\cdot\|)$ and then along the sequence $(N_{i})_{i\in \N}$ inject conditions which ensure the point does not belong to $\Bad_{d}(\varepsilon(1+\delta\varepsilon^{d}), \|\cdot\|)$. This is similar to ideas that appear in \cite{Bu03, BugMor2011, Moreira2018, KLWZ24, BandiDeSaxce, BakerWard} when studying sets of exact approximation order.
% \end{remark}

\subsection{Proof of Theorem~\ref{main}}

It is clear that in order to prove Theorem~\ref{main}, the following statement is needed.

\begin{theorem} \label{thm:lagrange winning} Fix
\begin{equation} \label{eq: beta parameter}
    0<\beta<\min\left\{\tfrac{\lambda_{d}(B(0,1))}{\lambda_{d-1}(B(0,1))}2^{-5d-4},  \left(\frac{c}{D_{d,\|\cdot\|_{2}}}\right)^{d+1} ,400^{-2(d+1)}\right\},
\end{equation}
where $c=\left((d!)2^{d}\lambda_{d}(B(0,1))\right)^{-\frac{1}{d}}$, and suppose that
\begin{equation} \label{eq: beta epsilon parameter}
    0<\varepsilon< c \beta^{6}.
\end{equation}
Pick parameters $(N_{i},R_{i})$ such that $N_{0}=0$ and
\begin{align} \label{eq: NR parameter}
     N_i&>\max\left\{\frac{1}{|\log \beta|} \left(\log \rho-\log c -(d+1)\log(\rho \min\left\{ 1, \varepsilon c^{-1} \beta^{-N_{i-1}-1} \right\} )\right), \left\lceil\beta^{-N_{i-1}}\right\rceil \right\}\\
     R_{i}&=R=\left\lceil 4-d+(d+1)\frac{-\log \varepsilon+\log c}{|\log \beta|} \right\rceil, \label{eq: Ri parameter} %\\
    %&\lim_{n\to\infty} \frac{nR}{N_{n}+1+n(R-1)}=0, \label{eq: tend to zero assumption}
\end{align}
and let $H\subset \R^{d}$ denote the set of affine hyperplanes. Then, the set
\begin{equation*}
    \left\{ \bx \in \R^{d}: \varepsilon \leq \Theta_{d}(\bx) \leq  \varepsilon\left(1+\beta^{-6d-2}2^{d}d!\lambda_{d}(B(0,1)) \varepsilon^{d}\right) \right\}
\end{equation*}
is $(N_{i},R_{i})$-rapid $(\beta, \rho, H)$-absolute winning.

\end{theorem}
\begin{remark}
    The bounds on $\beta>0$ and hence $\varepsilon>0$ are not optimised. The first term in \eqref{eq: beta parameter} is necessary to ensure Theorem~\ref{winning dimension} is applicable. The latter two terms could probably be improved. Note $R_{i}$ is independent of $i$ and is positive due to \eqref{eq: beta epsilon parameter} and so it is possible to pick a sequence $(N_{i})_{i \in \N}$ so that \eqref{eq: limit to zer} is satisfied.
\end{remark}

 \begin{proof}[Proof of Theorem~\ref{main} via Theorem~\ref{thm:lagrange winning}]
By Theorem~\ref{thm:lagrange winning} the set 
\begin{equation*}
    S(\varepsilon, \Delta(\beta)):=\left\{ \bx \in \R^{d}: \varepsilon \leq \Theta_{d}(\bx) \leq  \varepsilon\left(1+ \Delta(\beta)\varepsilon^{d}\right) \right\}
\end{equation*}
with
\begin{equation*}
    \Delta(\beta)=\beta^{-6d-2}2^{d}(d!\lambda_{d}(B(0,1)))
\end{equation*}
is winning for all pairs $(\beta, \varepsilon)$ satisfying \eqref{eq: beta parameter}, \eqref{eq: beta epsilon parameter}, with suitable choices of $(N_{i},R_{i})$. Applying Theorem~\ref{winning dimension} we see that
\begin{equation*}
    \dimh S(\varepsilon, \Delta(\beta)) \geq d- \frac{\log \left(c_{1}-c_{2} \beta\right)}{\log \beta},
\end{equation*}
where we have used that
\begin{equation*}
    \limsup_{n\to\infty} \frac{\sum_{i=1}^{n}R_{i}}{N_{n}+1+\sum_{i=1}^{n}(R_{i}-1)}\leq\limsup_{n\to \infty} \frac{nR}{\beta^{-n+1}+1+nR-n}=0.
\end{equation*}
 Note the bound holds for all $0<\varepsilon<c\beta^{6}$, and so taking
\begin{equation} \label{eq: delta bound}
    \delta=\Delta\left( \tfrac{1}{2}\min\left\{\tfrac{\lambda_{d}(B(0,1))}{\lambda_{d-1}(B(0,1))}2^{-5d-4},  \left(\frac{c}{D_{d,\|\cdot\|_{2}}}\right)^{d+1} ,400^{-2(d+1)}\right\} \right)
\end{equation}
we have that $\dimh S(\varepsilon, \delta)>0$ for all
\begin{equation}\label{eq: epsilon bound}
    0<\varepsilon<c \left(\tfrac{1}{2}\min\left\{\tfrac{\lambda_{d}(B(0,1))}{\lambda_{d-1}(B(0,1))}2^{-5d-4},  \left(\frac{c}{D_{d,\|\cdot\|_{2}}}\right)^{d+1} ,400^{-2(d+1)}\right\}\right)^{6}.
\end{equation}
The dimension statements follow readily by noting that $\dimh S(\varepsilon, \Delta(\beta))\to d$ as $\beta\to 0$.
 \end{proof}

% %%%%%%%%%%%%%%%%%%%%%%%%%%%%%%%%%
% %
% %
% %   PRELIMINARIES
% %
% %
% %%%%%%%%%%%%%%%%%%%%%%%%%%%%%%%%%

\section{Preliminaries for Theorem~\ref{winning dimension}} \label{section:prelims}
This section gives the definitions of Hausdorff, Packing, and Box dimension. We also prove the Box dimension statement appearing in Theorem~\ref{main}.

\subsection{Fractal Geometry} \label{section:fractal geometry}
In this section we recall various definitions of dimensions. See \cite{Falconer2013} for further details and properties. For $F\subset X$ a non-empty bounded set, for any $\rho>0$ the number $N_{\rho}(F)$ is the smallest number of balls, with radii equal to $\rho$, needed to cover $F$. Define
\begin{equation*}
    \dimub F=\limsup_{\rho\to 0} \frac{\log N_{\rho}(F)}{-\log \rho}\, , \quad \dimlb F=\liminf_{\rho\to 0} \frac{\log N_{\rho}(F)}{-\log \rho}\, ,
\end{equation*}
and, if both limits coincide, $\dimb F=\lim_{\rho\to 0} \tfrac{\log N_{\rho}(F)}{-\log \rho}$. Lastly, as introduced by Tricot \cite{Tricot1982} define the packing dimension, $\dimp$ as
\begin{equation*}
    \dimp F:= \inf\left\{\sup_{i} \dimub F_{i}: F\subset \bigcup_{i}F_{i} \right\}.
\end{equation*}
We have the following useful result on the lower bound of the packing dimension of sets.
\begin{lemma}[{\cite[Lemma 2.1(i)]{BishopPeres1996}}] \label{packing dim lower bound}
    Let $F\subset X$ be a closed subset. If,
for every open set $V$ intersecting $F$
\begin{equation*}
    \overline{\dim}_{B}(V\cap F)\geq s,
\end{equation*}
then $\dimp E\geq s$.
\end{lemma}

% Armed with the first part of Theorem~\ref{main} the statement on the Box dimension of $\cL_{d}(\|\cdot\|)$ follows easily. 
% \begin{proof}[Proof of Theorem~\ref{main}, \eqref{metric statement}]
% For any $r>0$,
% \begin{align*}
%     N_{\delta r^{d+1}}(\cL_{d}(\|\cdot\|_{2}))&\geq N_{\delta r^{d+1}}\left(B(0,r+\delta r^{d+1})\cap \cL_{d}(\|\cdot\|_{2})\right).
%     %\\    &\geq N_{\delta r^{d+1}}\left([0,r+\delta r^{d+1})\right),
% \end{align*}
% Since every interval of $$[0,r+\delta r^{d+1})=[0,\delta r^{d+1})\cup [\delta r^{d+1}, 2\delta r^{d+1})\cup \ldots \cup [r-\delta r^{d+1},r)\cup [r,r+\delta r^{d+1})$$ contains at least one point in $\cL_{d}(\|\cdot\|)$ by the first part of Theorem~\ref{main} and an interval of size $\delta r^{d+1}$ can intersect at most two of these, it follows that
% \begin{align*}
%     N_{\delta r^{d+1}}\left([0,r+\delta r^{d+1})\cap \cL_{d}(\|\cdot\|_{2})\right)&\geq \tfrac{1}{2}\frac{r}{\delta r^{d+1}}=\tfrac{1}{2}\delta^{-1}r^{-d}\\
%     & =\tfrac{1}{2}\delta^{-\tfrac{1}{d+1}} (\delta r^{d+1})^{-\tfrac{d}{d+1}}\, , 
% \end{align*}
% thus
% \begin{equation*}
%     \dimlb \cL_{d}(\|\cdot\|_{2})\geq \liminf_{r\to0}\frac{\log N_{\delta r^{d+1}}(\cL_{d}(\|\cdot\|_{2}))}{-\log \delta r^{d+1}}=\frac{d}{d+1}=1-\tfrac{1}{d+1}\, .
% \end{equation*}
% \end{proof}
The Hausdorff $s$-measure is defined as
\begin{equation*}
\cH^{s}(F)=\lim_{\rho \to 0^{+}}\inf\left\{ \sum_{i}|B_{i}|^{s} : \bigcup_{i} B_{i} \supseteq F \quad \text{ is countable, and }\quad |B_{i}|\leq \rho \right\}\, ,
\end{equation*}
and the Hausdorff dimension is defined to be
\begin{equation*}
\dimh F=\inf\left\{s\geq 0 : \cH^{s}(F)<\infty \right\}\, .
\end{equation*}
A standard tool to provide a lower bound on the Hausdorff dimension of a set is the Mass Distribution Principle, see for example \cite[Proposition 4.2]{Falconer2013}.
\begin{lemma*}[Mass Distribution Principle] \label{MDP}
Let $\mu$ be a probability measure supported on a subset $F \subseteq X$. Suppose that for $s>0$ there exists constants $c,\varepsilon>0$ such that
\begin{equation*}
\mu(B) \leq c |B|^{s}
\end{equation*}
for all balls $B \subset X$ with $|B|<\varepsilon$. Then $\dimh F \geq s$.
\end{lemma*}

\section{Proof of Theorem~\ref{winning dimension}} \label{winning dim proof}

Briefly, the proof is as follows: let us restrict Bobs legal moves at each step to $M$ disjoint different balls. By a volume argument using (AHL) we can see that $M\asymp \beta^{-\delta}$. By a volume argument using (DEC) we can see that Alice removes at most $k\asymp \beta^{\eta-\delta}$ of the $M$ options available to Bob in a standard turn (when $n \notin (N_{i})_{i\in \N}$). This is true at all times except along the sequence $(N_{i})_{i\in \N}$. At times $N_{i}$, Alice gives Bob precisely $1$ option from the $M$ possible balls for $R_{i}$ turns. From this idea we can construct a Cantor set with a ball $B$ in a typical level having $M-k$ children of size $\beta r(B)$, and at exceptional levels having exactly one child of size $\beta^{R_{i}}r(B)$. Constructing a natural measure on this Cantor set and applying the Mass Distribution Principle gives us our result.

    \begin{proof}[Proof of Theorem~\ref{winning dimension}] 
    Firstly, by the $\delta$-Ahlfors measure property we can use a volume argument to see that the number of disjoint balls contained in $B$ with centres in $X$ and radius $\beta r(B)$ is at least 
    \begin{equation*}
         \frac{\mu(\tfrac{1}{2}B)}{\mu(2\beta B))}\geq \left\lfloor \frac{a_{\mu}}{b_{\mu}}\left(\frac{1}{4}\right)^{\delta} \beta^{-\delta}\right\rfloor=:M\, .
    \end{equation*}
    To each ball $B\subset X$ let $\sigma_{1}(B), \ldots , \sigma_{M}(B)$ denote a generic disjoint collection of sub-balls contained in $B$ with radii $\beta r(B)$. Note these balls are in no fixed position and are free to vary in position between balls $B$.
   
    Let $B$ denote some generic turn of Bobs (not along the sequence $(N_{i})_{i\in \N}$). By a volume argument and the $(\eta, \cP)$-absolutely decaying property of $\mu$, the number of disjoint balls $B' \subseteq B$ of radius $\beta r(B)$ that can intersect $\Delta(P,\beta r(B))$ is at most
    \begin{align*}
            \frac{\mu\left(B\cap\Delta(P,2\beta r(B))\right)}{\mu(\beta B)} \leq C_{dec}\left(\frac{2\beta r(B)}{r(B)}\right)^{\eta}\frac{\mu(B)}{\mu(\beta B)}
            \leq \left\lceil C_{dec}\frac{b_{\mu}}{a_{\mu}}2^{\eta} \beta^{\eta-\delta}\right\rceil =:k\, .
    \end{align*}
    That is, by removing $\Delta(P,\beta r(B))$ Alice can stop Bob choosing at most $k$ balls from the set $(\sigma_{i}(B))_{i=1,\ldots, M}$. Thus Bob has at least $M-k$ balls to choose from on his next turn. Let $\sigma_{1}(B), \ldots, \sigma_{M-k}(B)$ denote the $M-k$ disjoint balls of radius $\beta r(B)$ that avoid the response by Alice to the ball $B$. Note the balls will depend on Alices response, but the important point is that we take exactly $M-k$ of them, they are contained in $B$, they avoid $\Delta(P,\beta r(B))$, and they are of radius $\beta r(B)$. 

    On turns along the sequence $(N_{i})_{i\in\N}$ Bob has only one choice of ball, the ball Alice forces him to pick. Let $B_{N_{i}}$ denote the ball Bob chose on his $N_{i}$th turn and $\sigma_{\infty_i}(B_{N_{i}})=B_{N_{i}+1}=A_{N_{i}}$ denote the ball Bob is forced to choose by Alice on his next turn. Let
    \begin{equation*}
         \cI_{n}=\begin{cases}
         \, \, \, \, \, \, \, \, \,   \{\infty_i\} \quad \quad \, \text{ if } \quad n=N_{i}+1 \, \,  \text{ for some }\, \,  i \in \N \, ,  \\
            \{1, \ldots , M-k\} \quad \text{ otherwise.} 
        \end{cases}
    \end{equation*}  
    Then, by setting $I_{0}=B_{0}$ and inductively, for each $B\in I_{n-1}$
    \begin{equation*}
        I_{n}(B):=\bigcup_{i \in \cI_{n}}\sigma_{i}(B) \quad \text{ and } \quad I_{n}=\bigcup_{B\in I_{n-1}} I_{n}(B)\, ,
    \end{equation*}
    we can define the Cantor set
    \begin{equation*}
        \cK:=\bigcap_{n\in \N} \bigcup_{B\in I_{n}}B\, .
    \end{equation*}
    Clearly every limit point in the Cantor set $\cK$ belongs to a legal sequence of Bobs moves, including the moves forced by Alice, in the $(N_{i},R_{i})$-rapid $(\beta,\rho,\cP)$-absolute game. Since $S$ is winning set we must have that $\cK\subseteq S$. \par 
    Calculating a lower bound Hausdorff dimension on the set $\cK$ is a standard application of the Mass Distribution Principle. For completeness we include it here. Construct a mass distribution $\nu$ on $\cK$ by setting $\nu(B_{0})=1$, and then inductively, supposing the mass has been distributed on the balls $B\in I_{n-1}$, for each $B'\in I_{n}(B)$ define
    \begin{equation*}
        \nu(B')=\frac{\nu(B)}{\# \cI_{n}}\, .
    \end{equation*}
    Clearly the measure is mass preserving and the support is $\cK$. The mass distribution defined on all levels of the Cantor set $\cK$  determines a Borel probability measure supported on $\cK$.
    In order to use the Mass Distribution principle we rewrite the measure on balls appearing in the Cantor set in terms of their radii. Using an inductive argument we see that for any ball $B\in I_{n}$ with $1\leq n <N_{1}$ we have $\nu(B)=(M-k)^{-n}$, and for $n\geq N_{1}$
    \begin{equation*}
        \nu(B)=
            (M-k)^{-(n-j_{n})} 
    \end{equation*}
    where $j_{n}:=\#\{i: N_{i}<n\}$. Note that 
    \begin{equation} \label{eq: radius per level}
        r(B)=\beta^{n-j_{n}+\sum_{i=1}^{j_{n}}R_{i}}r(B_{0}) \quad \text{ for any } \, \, B\in I_{n}.
    \end{equation} 
    Hence, for ball $B\in I_{n}$
\begin{equation*}
    \nu(B)=\left(\frac{r(B)}{r(B_{0})}\right)^{\left(1-\frac{\sum_{i=1}^{j_{n}}R_{i}}{n-j_{n}+\sum_{i=1}^{j_{n}}R_{i}}\right)\frac{\log(M-k)}{-\log \beta}}.
\end{equation*}
    Considering the exponent on the radius of $B$, it is strictly increasing as $n$ increases along $N_{j_{n}}\leq n < N_{j_{n}+1}$, and so the smallest possible exponent will be at $n=N_{j_{n}}+1$. Take
    \begin{equation*}
        s_{L}=\left(1-\frac{\sum_{i=1}^{L}R_{i}}{N_{L}+1+\sum_{i=1}^{L}(R_{i}-1)}\right)\frac{\log(M-k)}{-\log \beta}\, .
    \end{equation*}
    For every $s<\liminf_{L\to \infty}s_{L}$ we can choose $L$ sufficiently large so that $\inf_{i\geq L} s_{L}>s$. Then, for all $n>T_{L}$ and $B\in I_{n}$, we have
    \begin{equation*}
        \nu(B) \leq r(B_{0})^{-\inf_{i\geq L} s_{i}}r(B)^{\inf_{i\geq L} s_{i}}\leq C_{s} r(B)^{s}\, 
    \end{equation*}
    with $C_{s}>0$ the implied constant.
    Hence we have an upper bound on the measure of any ball with radius sufficiently small appearing in our construction of $\cK$.\par 
    Now consider a general ball with centre in $\cK$. That is, take any $0<r_{0}<\beta^{T_{L}}r(B_{0})$ and arbitrary ball $F\subset X$ with $r(F)<\tfrac{1}{2}r_{0}$ and some centre $c_{F}\in \cK$. There exists unique $m\in \N$ such that
    \begin{equation*}
        \beta^{m-j_{m}+\sum_{i=1}^{j_{m}}R_{i}}r(B_{0})\leq r(F) < \beta^{m-1-j_{m-1}+\sum_{i=1}^{j_{m-1}}R_{i}}r(B_{0}).
    \end{equation*}
    \begin{itemize}
    \item[-] If $j_{m}=j_{m-1}$ then by a volume argument we see that
    \begin{equation*}
        \#\{B' \in I_{m}: B'\cap F\neq \emptyset\} \leq \frac{\mu\left(B(c_{F},3\beta^{-1}r(B'))\right)}{\mu(B')}\leq \frac{b_{\mu}3^{\delta}r(B')^{\delta}\beta^{-\delta}}{a_{\mu}r(B')^{\delta}}\leq \tfrac{b_{\mu}}{a_{\mu}}(3\beta^{-1})^{\delta}
    \end{equation*}
    and so
    \begin{align*}
        \nu(F) &
        %\leq \mu(\widetilde{B})
        \leq \sum_{B'\in I_{m}\, : \, \,  B'\cap F\neq \emptyset} \nu(B') \leq \tfrac{b_{\mu}}{a_{\mu}}(3\beta^{-1})^{\delta} \nu(B') \leq \tfrac{b_{\mu}}{a_{\mu}}(3\beta^{-1})^{\delta}C_{s} r(B')^{s} \leq \tfrac{b_{\mu}}{a_{\mu}}(3\beta^{-1})^{\delta} C_{s}r(F)^{s}.
    \end{align*}
    
    \item[-]If $j_{m}\neq j_{m-1}$ then we must have $j_{m}=j_{m-1}+1$ and $m=N_{j_{m}}+1$. Consider the level $I_{m-1}$. Each of these balls has exactly one child in level $I_{m}$. By a volume argument
    \begin{equation*}
        \#\{B' \in I_{m-1}: B'\cap F\neq \emptyset\} \leq \frac{\mu\left(B(c_{F},3r(B')\right)}{\mu(B')}\leq \frac{b_{\mu}}{a_{\mu}}3^{\delta},
    \end{equation*}
    where we have used that $3r(B')\geq2r(B')+r(F)$. Thus $F$ intersects at most $\frac{b_{\mu}}{a_{\mu}}3^{\delta}$ in level $I_{m}$, and so 
    \begin{align*}
        \nu(F) &
        %\leq \mu(\widetilde{B})
        \leq \sum_{B'\in I_{m}\, : \, \,  B'\cap F\neq \emptyset} \nu(B') \leq \frac{b_{\mu}}{a_{\mu}}3^{\delta} C_{s} r(B')^{s} \leq  \frac{b_{\mu}}{a_{\mu}}3^{\delta}C_{s}r(F)^{s}
    \end{align*}
\end{itemize}

Note that in both instances the constant is independent of $F$, and so we have for any ball $F\subset B_{0}$ with radius sufficiently small
\begin{equation*}
    \nu(F)\leq C_{s,\mu,\beta}r(F)^{s}.
\end{equation*}

    We can pick $s<\liminf_{L\to\infty}s_{L}$ as arbitrary close as we like. Thus, by the Mass distribution principle
    \begin{equation*}
        \dimh S\cap B_{0} \geq \dimh \cK \geq \liminf_{L\to\infty} s_{L}=\liminf_{L\to\infty} \left(1-\frac{\sum_{i=1}^{L}R_{i}}{N_{L}+1+\sum_{i=1}^{L}(R_{i}-1)}\right)\frac{\log(M-k)}{-\log \beta}\, .
    \end{equation*}
    Noting
   \begin{equation} \label{eq: M-k size}
       M-k =\left\lfloor\beta^{-\delta}\frac{a_{\mu}}{b_{\mu}}\frac{1}{4^{\delta}}\right\rfloor-\left\lceil C_{dec}\frac{b_{\mu}}{a_{\mu}}2^{\eta}\beta^{\eta-\delta}\right\rceil\geq  \beta^{-\delta}\left(\frac{1}{2}\frac{a_{\mu}}{b_{\mu}}\frac{1}{4^{\delta}}-C_{dec}\frac{b_{\mu}}{a_{\mu}}2^{\eta+1}\beta^{\eta} \right)
   \end{equation}
   completes the proof of the Hausdorff dimension lower bound.

%\vspace{5cm}

We now prove the Packing dimension lower bound. Fix some $m\in\N$ and $B\in I_{m}$. For $n>m$, the number of descendents of $B$ in $I_{n}$ is
\begin{equation} \label{eq: number desc packing}
    (M-k)^{(n-j_{n})-(m-j_{m})}.
\end{equation}
For ease of notation let $r_{n}$ denote the radius of such ball. That is, $r_{n}:=\beta^{n-j_n+\sum_{i=1}^{j_n}R_i}r(B_0)$. The number of balls $B'\in I_{n}$ that can intersect a ball $F$ of radius $r_n$ is at most
\begin{equation*}
    \frac{\mu(B(c_F,3r_n))}{\mu(B')}
    \leq
    \frac{b_\mu}{a_\mu}3^\delta.
\end{equation*}
It follows from
\eqref{eq: number desc packing} that any cover needs at least
\begin{equation} \label{eq: covering number packing}
    N_{r_{n}}(B \cap \cK)\geq \frac{a_\mu}{b_\mu 3^\delta}
    (M-k)^{(n-j_{n})-(m-j_{m})}.
\end{equation}
Restrict to the levels $n\in (N_{i})_{i \in \N}$, immediately before an exceptional turn. Using that $j_{N_{i}}=i-1$ we have that $r_{N_{i}}=\beta^{N_{i}-i+1+\sum_{j=1}^{i-1}R_j}r(B_0)$. Therefore,
\begin{align}
    \overline{\dim}_{B}(\cK\cap B)
    &\overset{\text{\eqref{eq: covering number packing}}}{\geq}\limsup_{i\to\infty}\frac{(N_{i}-i+1-(m-j_{m}))\log(M-k)+\log \left( \frac{a_\mu}{b_\mu 3^\delta}\right)}{\left(N_{i}-i+1+\sum_{j=1}^{i-1}R_{j}\right)(-\log\beta)-\log(R(B_{0}))}\nonumber\\
    &=\limsup_{i\to\infty}\frac{N_{i}-i+1}{N_{i}-i+1+\sum_{j=1}^{i-1}R_{j}}\frac{\log(M-k)}{-\log\beta}\nonumber\\
    &=\limsup_{i\to\infty}\left(1-\frac{\sum_{j=1}^{i-1}R_{j}}{N_{i}-i+1+\sum_{j=1}^{i-1}R_{j}}\right)\frac{\log(M-k)}{-\log\beta}. \label{eq: lower bound for cylinder}
\end{align}
Note that the bound is independent of $m$ and choice of ball $B\in I_{m}$. Moreover, for any open set $V$ such that $V\cap\cK\neq \emptyset$ there exists point $x \in \cK$ and radius $r>0$ such that $B(x,r)\subset V$. Since $x\in \cK$ and $r_{n}\to0$ as $n\to \infty$, there exists some $m\in \N$ and $B\in I_{m}$ such that $x \in B$ and $B\subset V$, and so 
\begin{equation*}
    \dimub V\cap \cK \geq \dimub B\cap \cK \overset{\text{\eqref{eq: lower bound for cylinder}}}{\geq} \limsup_{i\to\infty}\left(1-\frac{\sum_{j=1}^{i-1}R_{j}}{N_{i}-i+1+\sum_{j=1}^{i-1}R_{j}}\right)\frac{\log(M-k)}{-\log\beta}. 
\end{equation*}
Therefore, by Lemma~\ref{packing dim lower bound} and using \eqref{eq: M-k size} we obtain the Packing dimension lower bound.
   \end{proof}

\section{Preliminaries for the proof of Theorem~\ref{thm:lagrange winning}} \label{sect:prelim2}

This section contains the necessary lemmas to prove Theorem~\ref{thm:lagrange winning}. We begin with the following well-known \textit{Simplex lemma}, introduced by Davenport and Schmidt (see for example \cite[p.57]{Schmidtbook1980}), which gives us a higher dimensional analogue of the notion that rational points repel each other in $\R$. The version presented below appears in \cite{KTV06}.

\begin{lemma*}[Simplex Lemma \cite{KTV06}] \label{simplex lemma}
Let $E\subset \R^{d}$ be a bounded convex body and suppose that
\begin{equation*}
    \lambda_{d}(E) \leq (d!)^{-1}N^{-(d+1)}\, .
\end{equation*}
Then all rational points $\tfrac{\bp}{q}\in \Q^{d}\cap E$ with $1\leq q< N$ lie on a rational affine hyperplane. 
\end{lemma*}

The next lemma enables us to say that if we are in a ball that has avoided neighbourhoods of rational points up to some denominator height $Q$ then we are guaranteed to contain a rational point with denominator not much larger than $Q$.

\begin{lemma} \label{lem:eventually see a rational}
Let $B\subset \R^{d}$ be a ball and suppose there exists $R_{0},R_{1}, \Gamma,\gamma>0$ such that
\begin{equation} \label{eq:empty level 1}
    B\cap \bigcup_{0< s < R_{0}} \bigcup_{\br\in \Z^{d}} B\left(\frac{\br}{s}, \Gamma \right) = \emptyset
\end{equation}
and 
\begin{equation} \label{eq:empty level 2}
    B\cap \bigcup_{R_{0} \leq  s < R_{1}} \bigcup_{\br\in \Z^{d}} B\left(\frac{\br}{s}, \gamma s^{-\tfrac{d+1}{d}} \right) = \emptyset .
\end{equation}
Then, there exists rational point $\frac{\bp}{q} \in \Q^{d}\cap \tfrac{3}{4}B$ with
\begin{equation*}
    R_{1}\leq q \leq D_{d,\|\cdot\|}^{d}\max\left\{ \Gamma^{-d}, R_{1}\gamma^{-d}, \left(\tfrac{1}{4}R_{1} r(B)\right)^{-d} \right\},
\end{equation*}
for constant $D_{d,\|\cdot\|}>0$ depending on the norm inducing the ball $B$ only.
\end{lemma}

\begin{proof}
    Dirichlet's Theorem states there exists constant $D_{d,\|\cdot\|}>0$ depending on the norm inducing $B$ only, such that for any $\bx \in \R^{d}$ and any $N\in \R_{+}$ there exists $\frac{\bp}{q}\in \Q^{d}$ such that
    \begin{equation} \label{eq:dirichlet thm}
        \left\| \bx -\tfrac{\bp}{q} \right\|< D_{d,\|\cdot\|} q^{-1} N^{-\tfrac{1}{d}} \quad \text{ and } \quad 1\leq q\leq N.
    \end{equation}
    Take any $\bx \in \tfrac{1}{2}B$, $N=D_{d,\|\cdot\|}^{d}\max\left\{ \Gamma^{-d}, R_{1}\gamma^{-d}, \left(\tfrac{1}{4}R_{1} r(B)\right)^{-d} \right\}$ and consider the solution $\tfrac{\bp}{q}\in \Q^{d}$ to \eqref{eq:dirichlet thm}.
    Observe that:
    \begin{itemize}
        \item[-] if $1\leq q < R_{0}$ then
        \begin{equation*}
            \left\| \bx -\tfrac{\bp}{q} \right\|< D_{d,\|\cdot\|} q^{-1} (D_{d,\|\cdot\|}^{d}\Gamma^{-d})^{-\tfrac{1}{d}} < \Gamma,
        \end{equation*}
        implying $B\cap B(\tfrac{\bp}{q}, \Gamma) \neq \emptyset$ which is false by \eqref{eq:empty level 1}.
        \item[-] if $R_{0}\leq q < R_{1}$ then $q^{-\tfrac{1}{d}} > R_{1}^{-\tfrac{1}{d}}$ and so
        \begin{equation*}
            \left\| \bx -\tfrac{\bp}{q} \right\|< D_{d,\|\cdot\|} q^{-1} (D_{d,\|\cdot\|}^{d}R_{1} \gamma^{-d})^{-\tfrac{1}{d}} < \gamma q^{-\tfrac{d+1}{d}},
        \end{equation*}
        implying $B\cap B(\tfrac{\bp}{q}, \gamma q^{-\tfrac{d+1}{d}})\neq \emptyset$ which is false by \eqref{eq:empty level 2}.
    \end{itemize}
    Therefore $R_{1}\leq q \leq N$. So
    \begin{equation*}
        \left\| \bx -\tfrac{\bp}{q} \right\|< D_{d,\|\cdot\|} q^{-1} D_{d,\|\cdot\|}^{-1}\tfrac{1}{4}R_{1} r(B)\leq \tfrac{1}{4} r(B),
    \end{equation*}
    and since $\bx \in \tfrac{1}{2}B$ the triangle inequality yields $\tfrac{\bp}{q}\in \tfrac{3}{4}B$.
\end{proof}

\subsection{A few geometric lemmas}

We will need the following easy geometric lemmas on Euclidean balls.

\begin{lemma} \label{lem: key geometric lem}
Let $\tfrac{1}{2}>A>B>0$. Take any $\by \in \partial B(0,1)$ and any affine hyperplane $\cH$ such that 
\begin{enumerate}[i)]
    \item \label{eq: y in hyper} $\by \in \cH$,
    \item \label{eq: hype too steep} $\inf_{z \in \cH\cap B(\by, A)} \|\bz\|_{2} >1-B$.
\end{enumerate}
Then, letting $\bv_{\cH}$ denote a unit vector normal to $\cH$ with sign so that $\langle \bv_{\cH}, \by\rangle \geq 0$, we have that
\begin{equation*}
    \langle \by, \bv_{\cH} \rangle \geq \frac{A-B}{A+B}.
\end{equation*}
\end{lemma}

\begin{proof}

Let $\bz=\langle \by, \bv_{\cH}\rangle \bv_{\cH}$. Note that
\begin{equation*}
    \|\by-\bz\|_{2}=(1-\langle \by, \bv_{\cH}\rangle^{2})^{1/2}.
\end{equation*}
If $\|\by-\bz\|_{2}<A$, then $\bz \in B(\by, A)$ and so by \eqref{eq: hype too steep}, 
\begin{equation*}
    \|\bz\|_{2}=\langle \by, \bv_{\cH}\rangle>1-B>\frac{A-B}{A+B}
\end{equation*}
in which case we are done. So assume $\|\by-\bz\|_{2}\geq A$. Then take
\begin{equation*}
    \bx=\left( 1-\frac{A}{\|\by-\bz\|_{2}}\right)\by +\frac{A}{\|\by-\bz\|_{2}}\bz\in \cH\cap B(\by,A).
\end{equation*}
Moreover, 
\begin{equation*}
    \|\bx\|_{2}\leq 1-\frac{A}{\|\by-\bz\|_{2}}(1-\langle\by, \bv_{\cH}\rangle)\leq 1-A\frac{1-\langle\by, \bv_{\cH}\rangle}{1+\langle\by, \bv_{\cH}\rangle}.
\end{equation*}
Hence, by \eqref{eq: hype too steep}, we have that
\begin{equation*}
        B>A\frac{1-\langle \by, \bv_{\cH} \rangle}{1+\langle \by, \bv_{\cH} \rangle },
\end{equation*}
and so 
\begin{equation*}
        \langle \by, \bv_{\cH}\rangle>\frac{A-B}{A+B}.
\end{equation*}
completing the proof. 
\end{proof}

\begin{lemma} \label{lem: colliding hyperplanes}
Fix some $\tfrac{1}{3}>C>0$. Let $\cH$ and $\cH'$ be two hyperplanes of $\R^{d}$ described by normal unit vectors $\bv, \bv'\in \R^{d}$. Let $\bx \in \cH$ and $r>0$. Then there exists point $\by\in \cH$ such that
\begin{equation*}
    B(\by, Cr)\subset B(\bx,r) \quad \text{ and } \quad B(\by, Cr)\cap \Delta(\cH', Cr)=\emptyset,
\end{equation*}
provided that
\begin{equation*}
    \sin \theta >\frac{2C}{1-C},
\end{equation*}
where $\theta\in (0,\tfrac{\pi}{2}]$ denotes the angle between $\bv$ and $\bv'$.
\end{lemma}

\begin{proof}
    Let $L= \cH \cap \cH'$. These intersect since $\theta \neq 0$. Note that $L$ partitions $\cH$. Let $\bn$ denote the unit vector orthogonal to $L$ lying inside $\cH$, and choose the sign so that $\bx=\bz+m \bn$ for some $m\geq 0$ and $\bz \in L$ (choose either sign if $m=0$). Then, provided there exists $t\in \R_{\geq m}$ such that
    \begin{equation} \label{eq: strict demand}
        t-m<(1-C)r \quad \text{ and } \quad t\sin \theta > 2Cr,
    \end{equation}
    then $\by=\bz+t\bn$ satisfies the lemma, since $$\|\by-\bx\|_{2}=t-m<(1-C)r$$
    so $B(\by,Cr)\subset B(\bx,r)$, and similarly $d(\by, \cH')= t\sin \theta > 2Cr$ so $B(\by, Cr) \cap \Delta(\cH',Cr) = \emptyset$.

    Clearly it is hardest to satisfy the equations \eqref{eq: strict demand} when $m=0$. In this case combining the two equations of \eqref{eq: strict demand} gives $\sin\theta > \tfrac{2C}{1-C}$.
\end{proof}

\section{Proof of Theorem~\ref{thm:lagrange winning}}
\label{sec: proof of main}

Set $\beta>0$ satisfying \eqref{eq: beta parameter} and take some $\varepsilon>0$ satisfying \eqref{eq: beta epsilon parameter}. Alice picks a sequence of pairs $(N_{i},R_{i})_{i \in \N}$ satisfying \eqref{eq: NR parameter}. To start the game Bob picks any ball $B_{0}$ with $r(B_{0})\geq \rho$. We now describe the two strategies Alice employs depending on whether $n\in (N_{i})_{i\in \N}$.

Throughout the proof $B_{i}$ will denote Bobs $i$th turn. Note by the rules of the game
\begin{equation*}
    r(B_{i})=\beta^{i+\sum_{j=1}^{k}(R_{j}-1)}r(B_{0}) \quad \text{ for } \quad N_{k}<i\leq N_{k+1}.
\end{equation*}
We will also set
\begin{equation*}
    Q_{i}:=(d!)^{-\tfrac{1}{d+1}} \lambda_{d}(2B_{i})^{-\tfrac{1}{d+1}}
\end{equation*}
throughout. On exceptional turns we also need intermediate balls. For each $i\in \N$ and $k\in \R_{+}$ let $\widetilde{B}_{i,k}(\bx)$ denote a ball with centre $\bx$ and radius
\begin{equation*}
    r(\widetilde{B}_{i,k})=\beta^{k}r(B_{N_{i}}).
\end{equation*}
Alongside these balls correspondingly define
\begin{equation*}
    \widetilde{Q}_{i,k}:= (d!)^{-\tfrac{1}{d+1}} \lambda_{d}(2\widetilde{B}_{i,k})^{-\tfrac{1}{d+1}}.
\end{equation*}
It is clear that
\begin{equation*}
    Q_{i+1}=\beta^{-\tfrac{d}{d+1}}Q_{i} \quad \text{for}\, \,  i\not\in (N_{i})_{i\in \N} \quad \text{ and } \quad \widetilde{Q}_{i,k}=\beta^{-k\tfrac{d}{d+1}}Q_{i},
\end{equation*}
so $\widetilde{Q}_{i,R_{i}}=Q_{N_{i}+1}$.

\subsection{Regular turns} Suppose $n\not\in (N_{i})_{i\in \N}$. By \eqref{eq: NR parameter} $0 \not \in (N_{i})_{i\in \N}$, so we begin with Bobs $0$th ball $B_{0}$. Alice applies the Simplex Lemma to the ball $2B_{0}$ to see that all rational points $\tfrac{\bp}{q}\in \Q^{d}\cap 2B_{0}$ with
\begin{equation*}
    1\leq q < Q_{0}
\end{equation*}
lie on a rational affine hyperplane, say $\cH_{0}$. Alice selects $A_{0}=\Delta\left(\cH_{0},\beta r(B_{0}) \right)$. Note that
\begin{equation} \label{eq: 0th level removal}
    A_{0} \supset B_{0} \cap\bigcup_{1\leq q< Q_{0}}\bigcup_{\bp\in \Z^{d}} B\left(\frac{\bp}{q}, \beta r(B_{0}) \right).
\end{equation}
For all other regular turns Alice's strategy is the same as above. That is, suppose Bob has chosen a ball $B_{n}$, then Alice will respond by applying the Simplex Lemma to $2B_{n}$ to see that all rational points $\tfrac{\bp}{q}\in \Q^{d}\cap 2B_{n}$ with
\begin{equation*}
    1\leq q < Q_{n}
\end{equation*}
lie on some rational affine hyperplane, say $\cH_{n}$. Alice concludes her turn by picking $A_{n}=\Delta\left(\cH_{n},\beta r(B_{n})\right)$. Again, note
\begin{equation*}
    A_{n}\supset B_{n} \cap \bigcup_{1\leq q< Q_{n}}\bigcup_{\bp\in \Z^{d}}B\left(\frac{\bp}{q}, \beta r(B_{n}) \right).
\end{equation*}
If at any point there are no rational points $\tfrac{\bp}{q}\in \Q^{d}\cap 2B_{n}$ with $q<Q_{n}$ Alice plays some arbitrary affine hyperplane. Set $$c_{1}=(d!)^{-\tfrac{1}{d}}2^{-1}\lambda_{d}(B(\mathbf{0},1))^{-\tfrac{1}{d}}$$
and note that 
\begin{equation*} \label{eq: Q_n def}
    r(B_{n})=c_{1} Q_{n}^{-\tfrac{d+1}{d}}=c_{1} \beta Q_{n-1}^{-\tfrac{d+1}{d}}, \qquad n-1 \not \in (N_{i})_{i\in \N}.
\end{equation*}
using that $r(B_{n-1})=\beta^{-1}r(B_{n})$ as per the rules of the game between regular turns. Hence, by the rules of the game, for any ball Bob chooses after Alice's turn we have
\begin{equation*}
    B_{n+1} \subset B_{n}\setminus \Delta(\cH_{n}, c_{1}\beta^{2} Q_{n-1}{^{-\frac{d+1}{d}}})
\end{equation*}
and so
\begin{equation} \label{eq: regular turn avoid rationals}
    B_{n+1}\cap \bigcup_{Q_{n-1}\leq q< Q_{n}}\bigcup_{\bp\in \Z^{d}}B\left( \tfrac{\bp}{q}, c_{1}\beta^{2}q^{-\frac{d+1}{d}}\right)=\emptyset, \qquad \text{ provided }\,  n\not \in (N_{i})_{i \in \N}. 
\end{equation}

\begin{remark}
    So far this is a standard argument to show that the set of $d$-dimensional badly approximable points is hyperplane absolute winning, see \cite[Theorem 2.5]{BroderickFishmanKleinbockReichWeiss2012} for more details. We include it here for completeness and notational purposes.
\end{remark}

\subsection{Exceptional turns} We now describe the strategy of Alice on turns when $n \in (N_{i})_{i\in \N}$. Suppose Bob has just played ball $B_{N_{i}}$. For each $i\in \N$ we assume that
\begin{equation} \label{eq: assumption 1}
    B_{N_{i}}\cap  \bigcup_{\widetilde{Q}_{i-1, R_{i}-1}\leq q < Q_{N_{i}-1}}\bigcup_{\bp\in \Z^{d}} B\left( \frac{\bp}{q}, c_{1}\beta^{2} q^{-\frac{d+1}{d}} \right) = \emptyset,
\end{equation}
and 
\begin{equation} \label{eq: assumption 2}
    B_{N_{i}}\cap  \bigcup_{1\leq q < \widetilde{Q}_{i-1,R_{i}-1}}\bigcup_{\bp\in \Z^{d}} B\left( \frac{\bp}{q}, \delta_{i-1} \right) = \emptyset,
\end{equation}
for some constant $\delta_{i-1}>0$ independent of $N_{i}$ to be defined. This assumption will be shown to be true at the end of the exceptional turn.

As our base case note this is true for $N_{1}$ if we set $N_{0}=0$. Indeed, we have that
\begin{equation*}
    B_{N_{1}}\cap \bigcup_{1\leq q < Q_{0}} \bigcup_{\bp\in \Z^{d}} B\left(\frac{\bp}{q}, \beta r(B_{0}) \right)= \emptyset \, , 
\end{equation*}
from the fact that $B_{N_{1}}\subset B_{1}$, that $B_{1}\cap A_{0}=\emptyset$ and \eqref{eq: 0th level removal}. So \eqref{eq: assumption 2} is satisfied if we set $\delta_{0}=\beta r(B_{0})$, which is independent of $N_{1}$. We also have \eqref{eq: assumption 1}. To see why suppose there exists some $\frac{\bp}{q}$ with $Q_{0}\leq q<Q_{N_{1}-1}$ such that
 \begin{equation*}
    B_{N_{1}}\cap B\left( \tfrac{\bp}{q}, c_{1}\beta^{2} q^{-\tfrac{d+1}{d}}\right) \neq \emptyset. 
 \end{equation*}
 There exists unique $1\leq m \leq N_{1}-1$ such that $Q_{m-1}\leq q < Q_{m}$. Since $B_{N_{1}}\subseteq B_{m+1}$ this implies
  \begin{equation*} \label{eq:contradiction}
    B_{m+1}\cap B\left( \tfrac{\bp}{q}, c_{1}\beta^{2} q^{-\tfrac{d+1}{d}}\right) \neq \emptyset. 
 \end{equation*}
 However, this is false by \eqref{eq: regular turn avoid rationals}, since $m\leq N_{1}-1$ so $m\not \in (N_{i})_{i \in \N}$. Therefore \eqref{eq: assumption 1} is true in the base case. 
\subsubsection{Finding a target surface} \label{sec: target surface}
We begin by finding a surface, denoted $S_{i}$, which will be composed of a collection of surfaces of the form $\partial B(\tfrac{\bp}{q},\varepsilon q^{-(d+1)/d})$. By assumption \eqref{eq: assumption 1}-\eqref{eq: assumption 2} we have that Lemma~\ref{lem:eventually see a rational} is applicable to $B_{N_{i}}$ with parameters
\begin{equation*}
    \Gamma=\delta_{i-1}, \quad \gamma=c_{1}\beta^{2}, \quad R_{0}=\widetilde{Q}_{i-1,R_{i-1}}, \quad R_{1}=Q_{N_{i}-1},
\end{equation*}
therefore, there exists some rational point $\frac{\bp}{q}\in \Q^{d}\cap \tfrac{3}{4}B_{N_{i}}$ with denominator
\begin{align*}
    Q_{N_{i}-1}\leq q < &D_{d,\|\cdot\|}^{d} \max \left\{ \delta_{i-1}^{-d}, Q_{N_{i}-1}c_{1}^{-d}\beta^{-2d}, 4^{d}Q_{N_{i}-1}^{-d} r(B_{N_{i}})^{-d} \right\} \\
    \leq & \widetilde{Q}_{i,2(d+1)}.
\end{align*}
The second inequality is because:
\begin{itemize}
    \item By our inductive assumption $\delta_{i-1}$ is independent of $N_{i}$. Using \eqref{eq: NR parameter} we see that $N_{i}$ is significantly larger than $N_{i-1}$ so that $$D_{d,\|\cdot\|}^{d}\delta_{i-1}^{-d}<Q_{N_{i}}<\widetilde{Q}_{i,2(d+1)}.$$

    \item Since $\beta<\left(\frac{c_{1}}{D_{d,\|\cdot\|_{2}}}\right)^{d+1}$, where $c=c_{1}$, by \eqref{eq: beta parameter} we have that
    \begin{equation*}
        D_{d,\|\cdot\|}^{d}Q_{N_{i}-1}\beta^{-2d}c_{1}^{-d} < Q_{N_{i}-1}\beta^{-2(d+1)\tfrac{d}{d+1}-\tfrac{d}{d+1}}=\widetilde{Q}_{i,2(d+1)}.
    \end{equation*}

    \item Note $r(B_{N_{i}})=c_{1} \beta Q_{N_{i}-1}^{-\tfrac{d+1}{d}}$ and by \eqref{eq: beta parameter} $\beta^{\tfrac{d}{d+1}}<4^{-d}$, so
    \begin{equation*}
        D_{d,\|\cdot\|}^{d}4^{d} Q^{-d}_{N_{i}-1}r(B_{N_{i}})^{-d}<\beta^{-\tfrac{d}{d+1}} Q_{N_{i}-1}c_{1}^{-d} \beta^{-d}<\beta^{-(d+3)\tfrac{d}{d+1}}Q_{N_{i}-1}< \widetilde{Q}_{i,2(d+1)}.
    \end{equation*}
\end{itemize}
Consider the set 
\begin{equation*}
    L_{0}:=\left\{ B\left(\tfrac{\bp}{q}, c_{1}\beta^{2} q^{-\tfrac{d+1}{d}}\right)\right\}_{\bp \in \Z^{d}, Q_{N_{i}-1}\leq q<Q_{N_{i}}}.
\end{equation*}
If there exists a ball from $L_{0}$, with rational centre say $\tfrac{\bp_{i}}{q_{i}}$, such that
\begin{equation} \label{eq:1}
    (1-3\beta)B_{N_{i}}\cap B\left(\tfrac{\bp_{i}}{q_{i}}, c_{1}\beta^{2} q^{-\tfrac{d+1}{d}}\right)\neq \emptyset
\end{equation}
then take the ball $$\widetilde{B}_{i,1}\left(\tfrac{\bp_{i}}{q_{i}}\right):=B\left(\tfrac{\bp_{i}}{q_{i}}, \beta r(B_{N_{i}})\right)\supseteq B\left(\tfrac{\bp_{i}}{q_{i}}, c_{1}\beta^{2} q^{-\tfrac{d+1}{d}}\right)\supset B\left(\tfrac{\bp_{i}}{q_{i}}, \varepsilon q^{-\tfrac{d+1}{d}}\right)$$
in the next step, with the last containment due to \eqref{eq: beta parameter}. Note that by the triangle inequality $\widetilde{B}_{i,1}\left(\tfrac{\bp_{i}}{q_{i}}\right)\subset B_{N_{i}}$ and \eqref{eq: assumption 1} we have
\begin{equation*}
    \widetilde{B}_{i,1}\left(\tfrac{\bp_{i}}{q_{i}}\right) \cap \bigcup_{\widetilde{Q}_{i-1,R_{i-1}}\leq q< Q_{N_{i}-1}} \bigcup_{\bp\in \Z^{d}}B\left(\tfrac{\bp}{q}, c_{1}\beta^{2} q^{-\tfrac{d+1}{d}}\right)=\emptyset
\end{equation*}
by \eqref{eq: assumption 1}. Suppose no such ball in $L_{0}$ satisfies \eqref{eq:1}. Then it follows that
\begin{equation*}
    (1-3\beta)B_{N_{i}} \cap \bigcup_{Q_{N_{i-1}}\leq q< Q_{N_{i}}} \bigcup_{\bp\in \Z^{d}}B\left(\tfrac{\bp}{q}, c_{1}\beta^{2} q^{-\tfrac{d+1}{d}}\right)=\emptyset.
\end{equation*}
Iterate the above process as follows; define for $k\geq 0$
\begin{equation*}
    L_{k+1}:=\left\{ B\left(\tfrac{\bp}{q}, c_{1}\beta^{2} q^{-\tfrac{d+1}{d}}\right) \right\}_{\bp\in \Z^{d}, \widetilde{Q}_{i,k}\leq q < \widetilde{Q}_{i,k+1}}
\end{equation*}
and suppose that 
\begin{equation*}
    \left(1-3\sum_{i=1}^{k+1}\beta^{i}\right)B_{N_{i}} \cap \bigcup_{Q_{N_{i-1}}\leq q< \widetilde{Q}_{i,k}} \bigcup_{\bp\in \Z^{d}}B\left(\tfrac{\bp}{q}, c_{1}\beta^{2} q^{-\tfrac{d+1}{d}}\right)=\emptyset.
\end{equation*}
If there exists a ball from $L_{k+1}$, with rational centre say $\tfrac{\bp_{i}}{q_{i}}$, such that
\begin{equation*}
    \left(1-3\sum_{i=1}^{k+1}\beta^{i}\right)B_{N_{i}}\cap B\left(\tfrac{\bp_{i}}{q_{i}}, c_{1}\beta^{2} q^{-\tfrac{d+1}{d}}\right)\neq \emptyset
\end{equation*}
then take the ball
$$\widetilde{B}_{i,k+1}\left(\tfrac{\bp_{i}}{q_{i}}\right):=B\left(\tfrac{\bp_{i}}{q_{i}}, \beta^{k+1} r(B_{N_{i}})\right)$$
in the next step. By the triangle inequality $\widetilde{B}_{i,k+1}\left(\tfrac{\bp_{i}}{q_{i}}\right)\subset \left(1-3\sum_{i=1}^{k}\beta^{i}\right)B_{N_{i}}$ and so
\begin{equation*}
    \widetilde{B}_{i,k+1}\left(\tfrac{\bp_{i}}{q_{i}}\right) \cap \bigcup_{Q_{N_{i-1}}\leq q< \widetilde{Q}_{i,k}} \bigcup_{\bp\in \Z^{d}}B\left(\tfrac{\bp}{q}, c_{1}\beta^{2} q^{-\tfrac{d+1}{d}}\right)=\emptyset.
\end{equation*}
This process must terminate. To see why, note that by \eqref{eq: beta parameter} we have $\beta<\frac{1}{13}$, and so
\begin{equation*}
    1-3\sum_{i=1}^{\infty}\beta^{i}>\frac{3}{4}.
\end{equation*}
The application of Lemma~\ref{lem:eventually see a rational} to $B_{N_{i}}$ told us there exists a rational $$\frac{\bp}{q}\in \frac{3}{4}B_{N_{i}} \subset \left(1-3\sum_{i=1}^{\infty}\beta^{i}\right)B_{N_{i}}$$ with $q\leq \widetilde{Q}_{i,2(d+1)}$.
Therefore, there exists some $0\leq k_{i} \leq 2(d+1)$ and $\tfrac{\bp_{i}}{q_{i}}\in \Q^{d}$ with $$\widetilde{Q}_{i,k_{i}-1}\leq q_{i}< \widetilde{Q}_{i,k_{i}}$$ such that the ball $\widetilde{B}_{i,k_{i}}\left(\tfrac{\bp_{i}}{q_{i}}\right)=B\left(\tfrac{\bp_{i}}{q_{i}}, \beta^{k_{i}}r(B_{N_{i}})\right)$ satisfies
\begin{align}
    \widetilde{B}_{i,k_{i}}\left(\tfrac{\bp_{i}}{q_{i}}\right)&\subset B_{N_{i}} \label{eq:contained in big ball}\\
    \widetilde{B}_{i,k_{i}}\left(\tfrac{\bp_{i}}{q_{i}}\right)&\supset B\left(\tfrac{\bp_{i}}{q_{i}}, c_{1}\beta^{2}q_{i}^{-\tfrac{d+1}{d}}\right) \nonumber\\
    \widetilde{B}_{i,k_{i}}\left(\tfrac{\bp_{i}}{q_{i}}\right) &\cap \bigcup_{Q_{N_{i-1}}\leq q <\widetilde{Q}_{i,k_{i}-1}} \bigcup_{\bp\in \Z^{d}} B\left( \tfrac{\bp}{q}, c_{1}\beta^{2} q^{-\tfrac{d+1}{d}}\right)=\emptyset. \label{eq: avoids rationals}
\end{align}
Note that for $k_{i}=0$ we have $\widetilde{Q}_{i,-1}=Q_{N_{i}-1}$.
Let $\ell_{i}\in \N$ denote the unique integer such that 
\begin{equation} \label{eq: ell size}
    \beta^{\ell_{i}+k_{i}}r(B_{N_{i}})> \varepsilon \widetilde{Q}_{i,k_{i}-1}^{-\tfrac{d+1}{d}}\geq \beta^{\ell_{i}+k_{i}+1}r(B_{N_{i}}).
\end{equation}
Note that \eqref{eq: beta epsilon parameter} ensures $\ell_{i}>4$. Using that 
    \begin{equation*}
    \widetilde{Q}_{i,k_{i}-1}^{-\tfrac{d+1}{d}}=2(d!\lambda_{d}(B(0,1)))^{\frac{1}{d}}\beta^{k_{i}-1}r(B_{N_{i}})
\end{equation*}
it can be seen that
\begin{equation} \label{eq: ell size explicit}
    \frac{-\log \varepsilon +\log c_{1}}{|\log \beta|}-2\leq \ell_{i}<\frac{-\log \varepsilon +\log c_{1}}{|\log \beta|}-1.
\end{equation}

Consider the ball 
\begin{equation} \label{eq: shrink still contains ball}
\widetilde{B}_{i,k_{i}+\ell_{i}}\left( \tfrac{\bp_{i}}{q_{i}}\right)=B(\tfrac{\bp_{i}}{q_{i}}, \beta^{k_{i}+\ell_{i}}r(B_{N_{i}})) \subset \widetilde{B}_{i,k_{i}}\left(\tfrac{\bp_{i}}{q_{i}}\right).
\end{equation}
We see that all rational points $\tfrac{\bp}{q}\in \Q^{d}$ with $\widetilde{Q}_{i,k_{i}-1}\leq q<\widetilde{Q}_{i,k_{i}+\ell_{i}}$ such that
\begin{equation} \label{eq: intersect?}
    B\left(\tfrac{\bp_{i}}{q_{i}}, \varepsilon q_{i}^{-\tfrac{d+1}{d}} \right) \cap B\left( \tfrac{\bp}{q}, \varepsilon q^{-\tfrac{d+1}{d}}\right) \neq \emptyset
\end{equation}
lie on a rational affine hyperplane, which we will call $\widetilde{\cH}_{i}$. 
To see why note that the bound on the denominators and \eqref{eq: intersect?} means the distance between $\tfrac{\bp_{i}}{q_{i}}$ and any intersecting rational from this range is at most 
$$2\varepsilon \widetilde{Q}_{i,k_{i}-1}^{-\tfrac{d+1}{d}}<2\beta^{k_{i}+\ell_{i}}r(B_{N_{i}}),$$
so are contained in $2\widetilde{B}_{i,k_{i}+\ell_{i}}\left( \tfrac{\bp_{i}}{q_{i}}\right)$. Applying the Simplex Lemma to $2\widetilde{B}_{i,k_{i}+\ell_{i}}\left( \tfrac{\bp_{i}}{q_{i}}\right)$ shows all rationals $\tfrac{\bp}{q}\in 2\widetilde{B}_{i,k_{i}+\ell_{i}}$ with $1\leq q< \widetilde{Q}_{i,k_{i}+\ell_{i}}$ lie on the affine hyperplane $\widetilde{\cH}_{i}$.

Let $\bv_{i}$ denote a unit vector normal to the hyperplane $\widetilde{\cH}_{i}$ and let
\begin{equation*}
    \widetilde{\cH}^{+}_{i}:=\left\{ \bx \in \R^{d}: \left\langle \bx-\tfrac{\bp_{i}}{q_{i}},\bv_{i} \right\rangle >0 \right\}, 
\end{equation*}
where $\langle\cdot,\cdot\rangle$ denotes the inner product.
Define our target surface to be
\begin{equation*}
    S_{i}:=2\widetilde{B}_{i,k_{i}+\ell_{i}}\left( \tfrac{\bp_{i}}{q_{i}}\right) \cap \widetilde{\cH}^{+}_{i} \cap \left( \partial \left( \bigcup_{\widetilde{Q}_{i,k_{i}-1}\leq q < \widetilde{Q}_{i,k_{i}+1}}\bigcup_{\bp\in \Z^{d}}B\left(\frac{\bp}{q}, \varepsilon q^{-\tfrac{d+1}{d}} \right) \right)\setminus \Delta\left(\widetilde{\cH}_{i}, \tfrac{1}{2}\varepsilon \widetilde{Q}_{i,k_{i}}^{-\tfrac{d+1}{d}} \right) \right).
\end{equation*}
See Figure~\ref{fig:Si+ image}. \\
\begin{figure}[h]
\centering
\begin{tikzpicture}[x=0.45cm,y=0.45cm,
  mainarc/.style={black!80,line width=1.5pt},
  guide/.style={gray!55,line width=.8pt},
  centre/.style={circle,fill=black,inner sep=1.4pt}
]
  % Clip to approximately the GeoGebra viewport.
  \clip (-14,-7.9) rectangle (10.5,9.9);

  % Main diagonal
  \draw[gray!55,line width=.85pt]
    (-12.7,{0.7410714286*(-12.7)+0.4419642857}) --
    (10.5,{0.7410714286*(10.5)+0.4419642857});

  % Parallel dashed
  \draw[gray!55,densely dashed,line width=.65pt]
    (-12.7,{0.7410714286*(-12.7)+2.3114285714}) --
    (10.5,{0.7410714286*(10.5)+2.3114285714});

  % Construction circles.
  \draw[guide] (-5.94,-3.96) circle[radius=3.78132252];
  \draw[guide] (-2.33084856,-1.28536098) circle[radius=2.51000259];
  \draw[guide] (-9.68067205,-6.73210518) circle[radius=0.98312769];
  \draw[guide] (3.91666855,3.34449545) circle[radius=1.91962170];
  \draw[gray!60,line width=.9pt] (-2.33084856,-1.28536098) circle[radius=10.68457558];
  \draw[guide] (0.78,1.02) circle[radius=2.80692002];

  % arcs forming S_i^+.
  \draw[mainarc] (-5.94,-3.96) ++(70.50105262:3.78132252)
      arc[start angle=70.50105262,end angle=193.13713150,radius=3.78132252];
  \draw[mainarc] (-2.33084856,-1.28536098) ++(82.90415724:2.51000259)
      arc[start angle=82.90415724,end angle=159.23671417,radius=2.51000259];
  \draw[mainarc] (0.78,1.02) ++(68.89180451:2.80692002)
      arc[start angle=68.89180451,end angle=176.21244367,radius=2.80692002];
  \draw[mainarc] (3.91666855,3.34449545) ++(88.02537689:1.91962170)
      arc[start angle=88.02537689,end angle=165.05679831,radius=1.91962170];

  % Centres on the diag
  \node[centre] at (-9.68,-6.73) {};
  \node[centre] at (-5.94,-3.96) {};
  \node[centre] at (-2.33,-1.28) {};
  \node[centre] at (0.78,1.02) {};
  \node[centre] at (3.91,3.34) {};

  % Radius
  \draw[->] (-2.33,-1.28) -- (0.16,-0.97);
  \draw[->] (6.91,5.56) -- (6.02,6.77);

  % Label
  \node[anchor=west,font=\Large] at (-4.8,3.1) {$S_i$};
  \node[anchor=north west,font=\small] at (-2.6,-1.5)
    {$\dfrac{\mathbf p_i}{q_i}$};
  \node[anchor=west,font=\small] at (-1.3,-1.6)
    {$\varepsilon q_i^{-\frac{d+1}{d}}$};
  \node[anchor=west,font=\large] at (7,7)
    {$\widetilde{\cH}_{i}$};
  \node[anchor=west,font=\large] at (4.6,-1.5)
    {$2\widetilde{B}_{i,k_i+\ell_i}$};
  \node[anchor=west,font=\small] at (6.2,5)
    {$\tfrac{1}{2}\varepsilon \widetilde{Q}_{i,k_i}^{-\frac{d+1}{d}}$};
\end{tikzpicture}
\caption{An example of a surface $S_{i}$ for $d=2$.}
\label{fig:Si+ image}
%\vspace{-1cm}
\end{figure}
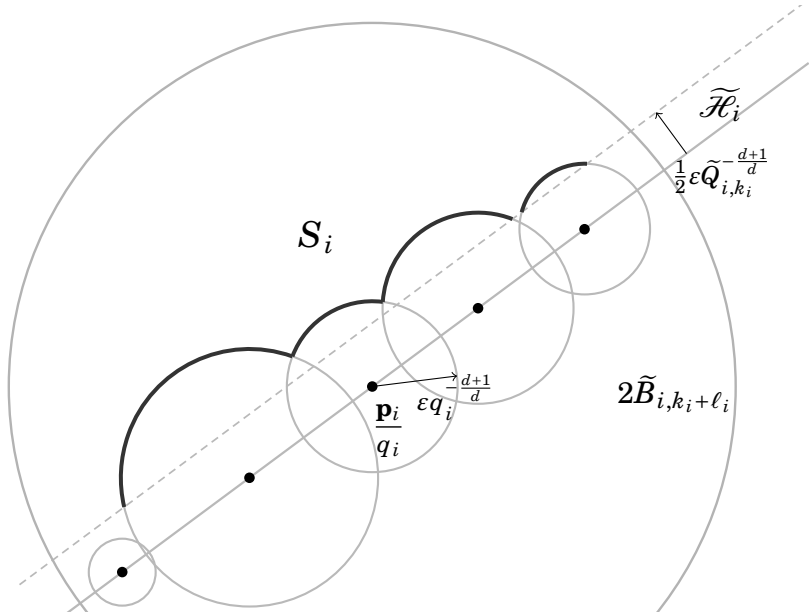

Before beginning with the approach towards the target surface we give some notation and properties that $S_{i}$ possesses. Firstly, note that 
\begin{equation} \label{eq: surface avoids rationals}
    \left(2\widetilde{B}_{i,k_{i}+\ell_{i}}\left( \tfrac{\bp_{i}}{q_{i}}\right)\setminus \Delta(\widetilde{\cH}_{i}, \tfrac{1}{2}\varepsilon \widetilde{Q}_{i,k_{i}}^{-\tfrac{d+1}{d}})\right) \cap \bigcup_{\widetilde{Q}_{i,k_{i}+1}\leq q < \widetilde{Q}_{i,k_{i}+\ell_{i}}} \bigcup_{\bp\in \Z^{d}} B\left( \tfrac{\bp}{q}, \varepsilon q^{-\tfrac{d+1}{d}} \right) = \emptyset,
\end{equation}
and so for any 
\begin{equation} \label{eq: r-height for S neighbourhood}
    0<r<\tfrac{1}{2}\varepsilon \widetilde{Q}_{i,k_{i}}^{-\tfrac{d+1}{d}}-\varepsilon \widetilde{Q}_{i,k_{i}+1}^{-\tfrac{d+1}{d}}=(\tfrac{1}{2}-\beta)\varepsilon \widetilde{Q}_{i,k_{i}}^{-\tfrac{d+1}{d}}
\end{equation}
we have that
\begin{equation} \label{eq: r-neighbourhood of S safe}
    \Delta(S_{i}, r)\cap 2\widetilde{B}_{i,k_{i}+\ell_{i}}\left( \tfrac{\bp_{i}}{q_{i}}\right)\cap \underset{\widetilde{Q}_{i,k_{i}+1}\leq q < \widetilde{Q}_{i,k_{i}+\ell_{i}}}{\bigcup_{Q_{N_{i-1}}\leq q < \widetilde{Q}_{i,k_{i}-1}}} \bigcup_{\bp\in \Z^{d}} B\left( \tfrac{\bp}{q}, \varepsilon q^{-\tfrac{d+1}{d}} \right) =\emptyset
\end{equation}
by \eqref{eq: surface avoids rationals} and \eqref{eq: avoids rationals}.
Let $\pi_{\widetilde{\cH}_{i}}:\R^{d}\to \widetilde{\cH}_{i}$ denote the orthogonal projection from $\R^{d}$ onto $\widetilde{\cH}_{i}$. Let
\begin{equation*}
    \cQ_{i}:=\left\{ \frac{\bp}{q} \in \Q^{d} \cap 2\widetilde{B}_{i,k_{i}+\ell_{i}}\left( \tfrac{\bp_{i}}{q_{i}}\right): \bp \in \Z^{d} \quad \text{ and } \quad \widetilde{Q}_{i,k_{i}-1}\leq q < 2\widetilde{Q}_{i,k_{i}} \right\}.
\end{equation*}
and for any $\by \in S_{i}$ let
\begin{equation*}
    \cQ_{i}(\by):=\left\{ \frac{\bp}{q} \in \cQ_{i}: \left\| \by-\tfrac{bp}{q}\right\|_{2}=\varepsilon q^{-\tfrac{d+1}{d}} \right\}.
\end{equation*}
For each $\tfrac{\bp}{q} \in \cQ_{i}$ let
\begin{equation*}
    U(\tfrac{\bp}{q})=\left(\partial B\left(\tfrac{\bp}{q}, \varepsilon q^{-\tfrac{d+1}{d}} \right) \setminus \Delta\left(\widetilde{\cH}_{i}, \tfrac{1}{2}\varepsilon \widetilde{Q}_{i,k_{i}}^{-\tfrac{d+1}{d}}\right) \right) \cap \widetilde{\cH}^{+}_{i}.
\end{equation*}
Note that some of these sets may be empty, particularly if $q>\widetilde{Q}_{i,k_{i}}$. Then 
\begin{equation} \label{eq: S covering}
    S_{i} \subseteq \bigcup_{\tfrac{\bp}{q} \in \cQ_{i}} U\left(\tfrac{\bp}{q} \right) \quad \text{ and } \quad \pi_{\widetilde{\cH}_{i}}(S_{i})= \pi_{\widetilde{\cH}_{i}}\left(\bigcup_{\tfrac{\bp}{q} \in \cQ_{i}} U\left(\tfrac{\bp}{q} \right) \right).
\end{equation}
That is, we can think of $S_{i}$ as a surface composed of a collection of spherical caps.
Since $\cQ_{i}\subset \widetilde{\cH}_{i}$ for every $\by \in \pi_{\widetilde{\cH}_{i}}(U(\tfrac{\bp}{q}))$ there exists unique $\bx\in U(\tfrac{\bp}{q})$ such that $\pi_{\widetilde{\cH}_{i}}(\bx)=\by$. Moreover,
\begin{equation*}
    \bx = \by +\xi(\by, \tfrac{\bp}{q}) \bv \quad \text{ where} \quad \xi(\by, \tfrac{\bp}{q}):=\left( \varepsilon^{2}q^{-2\tfrac{d+1}{d}}-\left\| \by-\tfrac{\bp}{q}\right\|_{2}^{2} \right)^{\tfrac{1}{2}}.
\end{equation*}
 Define the map $\Phi: \pi_{\widetilde{\cH}_{i}}(S_{i}) \to S_{i}$ to be
\begin{equation*}
    \Phi(\by)=\by+ \left( \sup_{\tfrac{\bp}{q}\in \cQ_{i}: \by \in \pi_{\widetilde{\cH}_{i}}(U(\tfrac{\bp}{q}))} \xi(\by,\tfrac{\bp}{q}) \right)\bv.
\end{equation*}

In later parts of the proof the properties of the tangent space to points in $S_{i}$ will be useful. Note that for any $\by \in S_{i}$ and any $\tfrac{\bp}{q} \in \cQ_{i}(\by)$ one can see that $\by-\tfrac{\bp}{q}$ is normal the tangent space $\cH_{\by}$ at $\by \in U(\tfrac{\bp}{q})$. Letting $\alpha$ denote the angle between $\by-\tfrac{\bp}{q}$ and $\bv_{i}$ we can see that

\begin{equation} \label{eq: sharp angle}
    \cos \alpha \geq\frac{\tfrac{1}{2}\varepsilon \widetilde{Q}_{i,k_{i}}^{-\tfrac{d+1}{d}}}{\varepsilon \widetilde{Q}_{i,k_{i}-1}^{-\tfrac{d+1}{d}}}=\tfrac{1}{2}\beta.
\end{equation}
In particular one can use this to see that for any $\by \in S_{i}$ the cone
\begin{equation} \label{eq: cone of avoidence}
    D_{\by}\cap S_{i}=\emptyset \quad \text{ for } \quad D_{\by}:= \left\{ \by +\bu: \|\bu\|_{2}>0 \quad \text{ and } \quad  \tfrac{\bu}{\|\bu\|_{2}}\cdot \bv_{i} \geq \left(1-\tfrac{1}{4}\beta^{2}\right)^{\frac{1}{2}} \right\}.
\end{equation}
Take the point $\omega_{i,1}:=\Phi(\tfrac{\bp_{i}}{q_{i}})$ and consider the ball
\begin{equation*}
    \widetilde{B}_{i,k_{i}+\ell_{i}+1+\tfrac{d}{d-1}}(\omega_{i,1})=B\left(\omega_{i,1},\beta^{k_{i}+\ell_{i}+1+\tfrac{d}{d-1}}r(B_{N_{i}})\right). 
\end{equation*}
For sake of notation, for $\bx \in \R^{d}$ and $t\in \R_{+}$ set
\begin{equation*}
    \widehat{B}_{i,t}(\bx)=\widetilde{B}_{i,k_{i}+\ell_{i}+1+t\tfrac{d}{d-1}}(\bx).
\end{equation*}
We have 
\begin{equation} \label{eq: Bhat 1 contained in Btilde}
    \widehat{B}_{i,1}(\omega_{i,1})\subset \widetilde{B}_{i,k_{i}+\ell_{i}},
\end{equation}
with the containment due to the triangle inequality since 
\begin{equation*}
    \left\| \omega_{i} - \tfrac{\bp_{i}}{q_{i}} \right\|_{2} \leq  \varepsilon \widetilde{Q}_{i,k_{i}-1}^{-\tfrac{d+1}{d}} \overset{\eqref{eq: ell size}}{<} \beta^{\ell_{i}+k_{i}}r(B_{N_{i}}) \quad \text{ and } \quad \beta^{k_{i}+\ell_{i}+1+\tfrac{d}{d-1}}r(B_{N_{i}}) < \beta^{k_{i}+\ell_{i}}r(B_{N_{i}}).
\end{equation*}
Furthermore note that 
\begin{equation} \label{eq: contained in S}
    \pi_{\widetilde{\cH}_{i}} (U(\tfrac{\bp_{i}}{q_{i}})) \supset \pi_{\widetilde{\cH}_{i}} (\widehat{B}_{i,1}(\omega_{i,1})).
\end{equation}
This follows from the observation that both projected balls share the same centrepoint, and that
\begin{align*}
    r\left( \pi_{\widetilde{\cH}_{i}} (\widehat{B}_{i,1}(\omega_{i,1})) \right)= \beta^{k_{i}+\ell_{i}+1+\tfrac{d}{d-1}}r(B_{N_{i}}) \overset{\eqref{eq: ell size}}{\leq} & \beta^{\tfrac{d}{d-1}}\varepsilon \widetilde{Q}_{i,k_{i}-1}^{-\tfrac{d+1}{d}}\\
    <&\beta^{\tfrac{1}{d-1}}\varepsilon \widetilde{Q}_{i,k_{i}}^{-\tfrac{d+1}{d}} \\
    \overset{\text{\eqref{eq: beta parameter}}}{\leq}&  \left( \varepsilon^{2}\widetilde{Q}_{i,k_{i}}^{-2\tfrac{d+1}{d}} - \tfrac{1}{4}\varepsilon^{2}\widetilde{Q}_{i,k_{i}}^{-2\tfrac{d+1}{d}} \right)^{\tfrac{1}{2}}  \qquad \\
    \leq& \left( \varepsilon^{2}q_{i}^{-2\tfrac{d+1}{d}} - \tfrac{1}{4}\varepsilon^{2}\widetilde{Q}_{i,k_{i}}^{-2\tfrac{d+1}{d}} \right)^{\tfrac{1}{2}} &&\text{($q_{i}<\widetilde{Q}_{i,k_{i}}$)}\\
    =& r\left( \pi_{\widetilde{\cH}_{i}} (U(\tfrac{\bp_{i}}{q_{i}}))\right)
\end{align*}
Using \eqref{eq: S covering} this implies that the only edges of the surface $S_{i}\cap \widehat{B}_{i,1}(\omega_{i,1})$ is $S_{i} \cup \partial \widehat{B}_{i,1}(\omega_{i,1})$. To see why note that the only edge points of $S_{i}$ are on $\Delta(\widetilde{\cH}_{i}, \tfrac{1}{2}\varepsilon \widetilde{Q}_{i,k_{i}}^{-(d+1)/d})$, which $U(\tfrac{\bp_{i}}{q_{i}})$ avoids by definition, and then use \eqref{eq: contained in S}.

To this end we can see that $S_{i}$ partitions $\widehat{B}_{i,1}(\omega_{i,1})$. Set
\begin{equation} \label{eq: S_i^+ def}
    S_{i}^{+}:=\widehat{B}_{i,1}(\omega_{i,1}) \setminus \left( \bigcup_{\widetilde{Q}_{i,k_{i}-1}\leq q < \widetilde{Q}_{i,k_{i}+1}}\bigcup_{\bp\in \Z^{d}}B\left(\frac{\bp}{q}, \varepsilon q^{-\tfrac{d+1}{d}} \right) \right).
\end{equation}
Then, we have
\begin{align} \label{eq: previous section result}
    \widehat{B}_{i,1}(\omega_{i,1})\cap S_{i}^{+}& \cap \bigcup_{Q_{N_{i-1}}\leq q < \widetilde{Q}_{i,k_{i}+\ell_{i}}} \bigcup_{\bp\in \Z^{d}} B\left( \tfrac{\bp}{q}, \varepsilon q^{-\tfrac{d+1}{d}} \right) =\emptyset.
\end{align}
This follows from the definition \eqref{eq: S_i^+ def} and \eqref{eq: r-neighbourhood of S safe}, where we are using $\omega_{i,1}\in S_{i}$, \eqref{eq: Bhat 1 contained in Btilde}, and 
\begin{equation*}
    r(\widehat{B}_{i,1}(\omega_{i,1}))< \beta^{\tfrac{1}{d-1}}\varepsilon \widetilde{Q}_{i,k_{i}}^{-\tfrac{d+1}{d}}\overset{\text{\eqref{eq: beta parameter}}}{<} (\tfrac{1}{2}-\beta)\varepsilon \widetilde{Q}_{i,k_{i}}^{-\tfrac{d+1}{d}}
\end{equation*}
so \eqref{eq: r-height for S neighbourhood} is satisfied and \eqref{eq: r-neighbourhood of S safe} is applicable. 

\subsubsection{Approaching the target surface}

We prove the following key lemma

\begin{lemma} \label{lem: spyglass lem}
    Suppose that
    \begin{equation} \label{eq: t size}
        %t_{i}\tfrac{1}{d-1}<\ell_{i}-(2+\tfrac{d}{d-1})\leq (t_{i}+1)\tfrac{1}{d-1}.
        t_{i}=(\ell_{i}-4)(d-1)
    \end{equation}
    Then there exists $\omega_{i,t}\in S_{i}$ such that
    \begin{align}
        \widehat{B}_{i,t_{i}}(\omega_{i,t_{i}})& \subset \widehat{B}_{i,1}(\omega_{i,1}) \label{lem: containment condition}\\
        \widehat{B}_{i,t_{i}}(\omega_{i,t_{i}})&\cap S_{i}^{+} \cap \bigcup_{Q_{N_{i-1}}\leq q < \widetilde{Q}_{i,k_{i}+\ell_{i}+t_{i}-1}} \bigcup_{\bp\in \Z^{d}} B\left( \tfrac{\bp}{q}, \varepsilon q^{-\tfrac{d+1}{d}}\right)=\emptyset\, .\label{lem: avoid rationals condition}
    \end{align}
\end{lemma}

\begin{remark}
    One can think of this as playing the hyperplane absolute winning game on $S_{i}$. If $B\cap S_{i}$ is not contained in a $\beta^{\frac{d}{d-1}}r(B)$-neighbourhood of a hyperplane then the proof is similar to the regular turn strategy. When $B\cap S_{i}$ is contained in a neighbourhood of some hyperplane the proof is harder. The key idea is to use the fact that we know rational points $\tfrac{\bp}{q}$ are close to $S_{i}$. We stretch out cylinders from these rationals onto $S_{i}$ and then use the Simplex lemma to remove affine hyperplanes containing dangerous regions. Provided the cylinders are `tangential enough' to the surface $S_{i}$, and are thicker than the neighbourhood of the affine hyperplane we need to remove then we know these affine hyperplanes cannot remove too much.
\end{remark}

\begin{proof}
The method of proof is to iteratively construct balls
\begin{equation}\label{eq: nested iterative assumption}
\widehat{B}_{i,1}(\omega_{i,1}) \supset \widehat{B}_{i,2}(\omega_{i,2}) \supset \widehat{B}_{i,3}(\omega_{i,3}) \supset \cdots \supset \widehat{B}_{i,t_{i}}(\omega_{i,t_{i}})
\end{equation}
such that for each $1\leq j \leq t_{i}$
\begin{equation} \label{eq: avoid rational iterative assumption}
    \widehat{B}_{i,j}(\omega_{i,j})\cap S_{i}^{+} \cap \bigcup_{Q_{N_{i-1}}\leq q < \widetilde{Q}_{i,k_{i}+\ell_{i}+j-1}}\bigcup_{\bp\in \Z^{d}} B\left( \tfrac{\bp}{q}, \varepsilon q^{-\tfrac{d+1}{d}} \right)=\emptyset.
\end{equation}
The base case follows from the previous section, namely \eqref{eq: previous section result} Let us assume we have arrived at some 
\begin{equation*}
\widehat{B}_{i,j}(\omega_{i,j})\subset \widehat{B}_{i,j-1}(\omega_{i,j-1}) \subset \cdots \subset \widehat{B}_{i,1}(\omega_{i,1}),    
\end{equation*}
with $\omega_{i,u}\in S_{i}$ for $u=1, \ldots j\leq t_{i}-1$, such that \eqref{eq: avoid rational iterative assumption} is satisfied. We will now construct ball $\widehat{B}_{i,j+1}(\omega_{i,j+1})$ for some $\omega_{i,j+1}\in S_{i}$.\par 
The strategy is split into two cases of increasing difficulty:
\begin{itemize}
\item \textit{Suppose that for all affine hyperplanes $\cH^{*}$
\begin{equation} \label{eq: non-hyperplane covering}
    \left((1-\beta)\widehat{B}_{i,j}(\omega_{i,j}) \cap S_{i} \right) \setminus \Delta\left(\cH^{*}, 3\beta^{k_{i}+\ell_{i}+1+(j+1)\tfrac{d}{d-1}}r(B_{N_{i}}) \right) \neq \emptyset.
\end{equation}}
Then apply the Simplex lemma to $2\widehat{B}_{i,j}(\omega_{i,j})$ to see that all rational points $\tfrac{\bp}{q}\in \Q^{d}\cap 2\widehat{B}_{i,j}(\omega_{i,j})$ with
\begin{equation*}
    1\leq q < \widetilde{Q}_{i,k_{i}+\ell_{i}+1+j\tfrac{d}{d-1}}
\end{equation*}
lie on a rational affine hyperplane, say $\cH^{*}$. Clearly this includes the rational points with denominators up to $\widetilde{Q}_{i,k_{i}+\ell_{i}+j}$. Note that
\begin{equation} \label{eq: big hype cover little hype}
    \Delta\left(\cH^{*}, \beta^{k_{i}+\ell_{i}+1+(j+1)\tfrac{d}{d-1}}r(B_{N_{i}})\right) \supseteq \Delta\left(\cH^{*}, \varepsilon \widetilde{Q}_{i,k_{i}+\ell_{i}+j-1}^{-\tfrac{d+1}{d}} \right)
\end{equation}
since
\begin{align*}
    \varepsilon \widetilde{Q}_{i,k_{i}+\ell_{i}+j-1}^{-\tfrac{d+1}{d}}&\overset{\text{\eqref{eq: ell size}}}{<}\beta^{k_{i}+\ell_{i}+\ell_{i}+j}r(B_{N_{i}}) \overset{\text{\eqref{eq: t size}}}{<}\beta^{k_{i}+\ell_{i}+1+\tfrac{d}{d-1}+(j+1)\tfrac{1}{d-1}+j}r(B_{N_{i}})\\
    &<\beta^{k_{i}+\ell_{i}+1+(j+1)\tfrac{d}{d-1}}r(B_{N_{i}}).
\end{align*}
Moreover
\begin{equation} \label{eq: little hype remove bad}
    \Delta\left(\cH^{*}, \varepsilon \widetilde{Q}_{i,k_{i}+\ell_{i}+j-1}^{-\tfrac{d+1}{d}} \right) \supseteq \widehat{B}_{i,j}(\omega_{i,j})\cap S_{i}^{+}\cap \bigcup_{\widetilde{Q}_{i,k_{i}+\ell_{i}+j-1}\leq q < \widetilde{Q}_{i,k_{i}+\ell_{i}+j}}\bigcup_{\bp\in \Z^{d}}B\left( \frac{\bp}{q}, \varepsilon q^{-\tfrac{d+1}{d}} \right).
\end{equation}
Since we assumed \eqref{eq: non-hyperplane covering} there exists some 
\begin{equation*}
    \omega_{i,j+1} \in S_{i}\cap (1-\beta)\widehat{B}_{i,j}(\omega_{i,j}) \setminus \Delta\left(\cH^{*}, 3\beta^{k_{i}+\ell_{i}+1+(j+1)\tfrac{d}{d-1}}r(B_{N_{i}})\right).
\end{equation*}
Then, by the triangle inequality we see that
\begin{equation*}
    \widehat{B}_{i,j+1}(\omega_{i,j+1})\subset \widehat{B}_{i,j}(\omega_{i,j}) \setminus \Delta\left( \cH^{*}, \beta^{k_{i}+\ell_{i}+1+(j+1)\tfrac{d}{d-1}}r(B_{N_{i}}) \right).
\end{equation*}
Hence \eqref{eq: nested iterative assumption} is satisfied. By \eqref{eq: big hype cover little hype} and \eqref{eq: little hype remove bad} we have
\begin{equation} \label{eq: disjoint up to some Q case 1}
    \widehat{B}_{i,j+1}(\omega_{i,j+1}) \cap \bigcup_{\widetilde{Q}_{i,k_{i}+\ell_{i}+j-1}\leq q < \widetilde{Q}_{i,k_{i}+\ell_{i}+j}}\bigcup_{\bp\in \Z^{d}}B\left( \frac{\bp}{q}, \varepsilon q^{-\tfrac{d+1}{d}} \right)=\emptyset.
\end{equation}
Combining \eqref{eq: disjoint up to some Q case 1} with our inductive assumption \eqref{eq: avoid rational iterative assumption} we obtain
\begin{equation*}
    \widehat{B}_{i,j+1}(\omega_{i,j+1})\cap S_{i}^{+} \cap \bigcup_{Q_{N_{i-1}}\leq q < \widetilde{Q}_{i,k_{i}+\ell_{i}+j}}\bigcup_{\bp\in \Z^{d}}B\left( \tfrac{\bp}{q},\varepsilon q^{-\tfrac{d+1}{d}}\right)=\emptyset.
\end{equation*}
This completes the first case.

\item \textit{There exists rational affine hyperplane $\cH^{**}$ such that
\begin{equation*}
    (1-\beta)\widehat{B}_{i,j}(\omega_{i,j}) \cap S_{i} \subset \Delta\left(\cH^{**}, 3\beta^{k_{i}+\ell_{i}+1+(j+1)\tfrac{d}{d-1}}r(B_{N_{i}}) \right).
\end{equation*}}
Suppose $\cH^{**}$ is an affine hyperplane satisfying the above equation. We make a minor adjustment to the affine hyperplane, by translating it so that $\omega_{i,j}$ is contained in the hyperplane. Letting $\cH^{*}$ such hyperplane it is easy to see that
\begin{equation}\label{eq: hyperplane covering surface}
    (1-\beta)\widehat{B}_{i,j}(\omega_{i,j}) \cap S_{i} \subset \Delta\left(\cH^{*}, 9\beta^{k_{i}+\ell_{i}+1+(j+1)\tfrac{d}{d-1}}r(B_{N_{i}}) \right).
\end{equation}
Let $\bv_{*}$ denote the unit vector normal to $\cH^{*}$ such that $\langle \bv_{*}, \bv_{i}\rangle \geq 0$. We prove the following claim

\textbf{Claim:} Let $\tfrac{\bp_{i,j}}{q_{i,j}}\in \cQ_{i}$ be such that $$\left\|\omega_{i,j}-\tfrac{\bp_{i,j}}{q_{i,j}}\right\|_{2}=\varepsilon q_{i,j}^{-\tfrac{d+1}{d}}.$$
Then
\begin{equation}\label{eq: angle target}
    \frac{\left\langle \omega_{i,j} -\tfrac{\bp_{i,j}}{q_{i,j}}, \bv_{*} \right\rangle}{\varepsilon q_{i,j}^{-\tfrac{d+1}{d}}}\geq \frac{1-40\beta^{\frac{d}{d-1}}}{1+40\beta^{\frac{d}{d-1}}}.
\end{equation}
\begin{proof}[Proof of Claim:]
We will apply Lemma~\ref{lem: key geometric lem} by replacing the unit ball with $\partial B\left(\frac{\bp_{i,j}}{q_{i,j}}, \varepsilon q_{i,j}^{-\tfrac{d+1}{d}} \right)$, taking the centre $\by=\omega_{i,j}$ and setting
\begin{equation*}
    A=\tfrac{\tfrac{1}{4}r(\widehat{B}_{i,j}(\omega_{i,j}))}{\varepsilon q_{i,j}^{-\tfrac{d+1}{d}}}, \quad B=\tfrac{10 \beta^{\frac{d}{d-1}}r(\widehat{B}_{i,j}(\omega_{i,j}))}{\varepsilon q_{i,j}^{-\tfrac{d+1}{d}}}.
\end{equation*}
Note $0<B<A<\frac{1}{2}$ by \eqref{eq: beta parameter}. Hence, provided
\begin{equation}
    \label{key}
    \inf_{\bz \in \cH^{*} \cap \frac{1}{4}\widehat{B}_{i,j}(\omega_{i,j})}\left\|\bz -\tfrac{\bp_{i,j}}{q_{i,j}}\right\|_{2}>\varepsilon q_{i,j}^{-\tfrac{d+1}{d}}\left(1- \tfrac{10 \beta^{\frac{d}{d-1}}r(\widehat{B}_{i,j}(\omega_{i,j}))}{\varepsilon q_{i,j}^{-\tfrac{d+1}{d}}} \right),
\end{equation}
\eqref{eq: angle target} follows. Therefore, suppose there exists $\bz \in \cH^{*} \cap \frac{1}{4}\widehat{B}_{i,j}(\omega_{i,j})$ such that
\begin{equation} \label{eq: key contradiction}
    \left\|\bz -\tfrac{\bp_{i,j}}{q_{i,j}}\right\|_{2}<\varepsilon q_{i,j}^{-\tfrac{d+1}{d}}\left(1- \tfrac{10 \beta^{\frac{d}{d-1}}r(\widehat{B}_{i,j}(\omega_{i,j}))}{\varepsilon q_{i,j}^{-\tfrac{d+1}{d}}} \right).
\end{equation}
The idea is to construct a path $P$ satisfying
\begin{equation*}
    P\subset \frac{1}{2}\widehat{B}_{i,j}(\omega_{i,j}) \setminus \Delta(\cH^{*}, 9\beta^{\frac{d}{d-1}}r(\widehat{B}_{i,j}(\omega_{i,j})))
\end{equation*}
with one endpoint below $S_{i}$ and the other endpoint above $S_{i}$. Since $S_{i}$ is from the continuous graph $\Phi$ and partitions the relevant ball, we would then have that $P\cap S_{i} \neq \emptyset$, which would contradict our assumption \eqref{eq: hyperplane covering surface}. Hence \eqref{eq: key contradiction} is false and therefore \eqref{key} follows completing the proof.
By \eqref{eq: key contradiction} it follows that
\begin{equation*}
    B\left(\bz, 10 \beta^{\frac{d}{d-1}}r(\widehat{B}_{i,j}(\omega_{i,j}))\right) \subset B\left(\frac{\bp_{i,j}}{q_{i,j}}, \varepsilon q_{i,j}^{-\tfrac{d+1}{d}} \right).
\end{equation*}
Set
\begin{align*}
    \bx=\bz+10 \beta^{\frac{d}{d-1}}r(\widehat{B}_{i,j}(\omega_{i,j}))\bv_{*} \in &B\left( \omega_{i,j},\frac{1}{4}r(\widehat{B}_{i,j}(\omega_{i,j}))+10 \beta^{\frac{d}{d-1}}r(\widehat{B}_{i,j}(\omega_{i,j})) \right)\\ &\subset \tfrac{1}{2}\widehat{B}_{i,j}(\omega_{i,j}),
\end{align*}
Observe that there exists 
\begin{equation*}
    \bs_{1}=\pi_{\widetilde{\cH}_{i}}(\bx)+\xi' \bv_{i} \in \partial B\left(\frac{\bp_{i,j}}{q_{i,j}}, \varepsilon q_{i,j}^{-\tfrac{d+1}{d}} \right) \setminus \Delta(\cH^{*}, 9 \beta^{\frac{d}{d-1}}r(\widehat{B}_{i,j}(\omega_{i,j}))
\end{equation*}
for some $\xi'>0$. If $\bs_{1}\not \in S_{i}$ then there exists some $\xi>\xi'>0$ such that 
\begin{equation*} \label{eq: a contradition1}
\bs=\pi_{\widetilde{\cH}_{i}}(\bx)+\xi \bv_{i} \in S_{i}\setminus \Delta(\cH^{*}, 9 \beta^{\frac{d}{d-1}}r(\widehat{B}_{i,j}(\omega_{i,j})).
\end{equation*}
In either case $\bx$ lies below some $\bs \in S_{i}$, that is 
\begin{equation} \label{eq: above x}
\langle \bx,\bv_{i}\rangle<\langle \bs,\bv_{i}\rangle.
\end{equation}
The point $\bx$ will be our endpoint of $P$ below $S_{i}$. Decompose $\bv_{i}$ so that
\begin{equation*}
    \bv_{i}=\langle \bv_{i}, \bv_{*}\rangle \bv_{*}+ \hat{\bv},
\end{equation*}
for some $\hat{\bv}$ parallel to $\cH^{*}$. Letting $\theta$ be the angle between $\bv_{i}$ and $\bv_{*}$ we have that
\begin{align}
    \|\hat{\bv}\|_{2}&=\sin\theta \label{eq: length hat}\\
    \|\pi_{\widetilde{\cH}_{i}}(\hat{\bv})\|_{2}&=\sin \theta \cos \theta \label{eq: proj length hat}\\
    \|\pi_{\widetilde{\cH}_{i}}(\bv_{*})\|_{2}&=\sin \theta \label{eq: proj length star}\\
    \langle \bv_{i}, \hat{\bv}\rangle&=(\sin\theta)^{2} \label{eq: hat height}
\end{align}
We now split into two cases depending on the size of the angle $\theta$.

Firstly suppose that
\begin{equation} \label{case1 theta small}
    \tan \theta \leq \frac{1}{40} \beta^{-\frac{d}{d-1}}.
\end{equation}
Then $\cos \theta>0$. Let
\begin{equation*}
    \by=\omega_{i,j}+\frac{10\beta^{\frac{d}{d-1}}r(\widehat{B}_{i,j}(\omega_{i,j}))}{\langle \bv_{i},\bv_{*}\rangle}\bv_{i}.
\end{equation*}
Note that the vector $\by-\bx$ is parallel to $\cH^{*}$. Set $\bv_{\bz}= \frac{\by-\bx}{\|\by-\bx\|_{2}}$ and consider the segment
\begin{align*}
    L&:=\left\{ L(t):=\bx +t \bv_{\bz}: 0 \leq t \leq \|\by-\bx\|_{2}\right\}
\end{align*}
Since $\bv_{z}$ is parallel to $\cH^{*}$ and $\bx \not\in \Delta(\cH^{*}, 9 \beta^{\frac{d}{d-1}}r(\widehat{B}_{i,j}(\omega_{i,j}))$ we have that
\begin{equation*}
    L\cap \Delta(\cH^{*}, 9 \beta^{\frac{d}{d-1}}r(\widehat{B}_{i,j}(\omega_{i,j})) =\emptyset.
\end{equation*}
Take
\begin{align*} \label{eq: t_0 size}
    t_{0}&= \|\by-\bx\|_{2}\nonumber \\
    &\leq \|\bz-\omega_{i,j}\|_{2}+ 10\beta^{\frac{d}{d-1}}r(\widehat{B}_{i,j}(\omega_{i,j})) \frac{\|\langle \bv_{i},\bv_{*}\rangle \bv_{*}-\bv_{i}\|_{2}}{\langle \bv_{i},\bv_{*}\rangle}\nonumber\\
     &\overset{\text{\eqref{eq: length hat}}}{\leq} \|\bz-\omega_{i,j}\|_{2}+10\beta^{\frac{d}{d-1}}r(\widehat{B}_{i,j}(\omega_{i,j})) \tan\theta.\nonumber\\
     &\overset{\text{\eqref{case1 theta small}}}{\leq} \frac{1}{2}r(\widehat{B}_{i,j}(\omega_{i,j})) 
\end{align*}
and so $L\subset \tfrac{1}{2}\widehat{B}_{i,j}(\omega_{i,j})$. Hence
\begin{equation*}
    L \subset \frac{1}{2}\widehat{B}_{i,j}(\omega_{i,j}) \setminus \Delta(\cH^{*}, 9\beta^{\frac{d}{d-1}}r(\widehat{B}_{i,j}(\omega_{i,j}))).
\end{equation*}
Note that $L(t_{0})=\by\in L$ and $\langle \by, \bv_{i} \rangle>\langle \omega_{i,j},\bv_{i}\rangle$ so $L(t_{0})$ lies above $S_{i}$. Since $L(0)=\bx$, combining this with \eqref{eq: above x} we see that $L$ satisfies the properties of the path $P$ and so we are done in the case of \eqref{case1 theta small}. Now suppose that 
\begin{equation} \label{case2 theta large}
    \tan \theta > \frac{1}{40} \beta^{-\frac{d}{d-1}}.
\end{equation}
We know that for any $\bu\in S_{i}$ with
\begin{equation*}
    \|\pi_{\widetilde{\cH}_{i}}(\omega_{i,j}-\bu)\|_{2}\leq 20\beta^{\frac{d}{d-1}}r(\widehat{B}_{i,j}(\omega_{i,j}))
\end{equation*}
that
\begin{equation} \label{eq: s height}
   \langle \bu-\omega_{i,j}, \bv_{i}\rangle
   \overset{\text{\eqref{eq: sharp angle}}}{\leq} 20\beta^{\frac{d}{d-1}}r(\widehat{B}_{i,j}(\omega_{i,j})) \frac{(1-\frac{1}{4}\beta^{2})^{\frac{1}{2}}}{\frac{1}{2}\beta} \leq 41\beta^{\frac{1}{d-1}}r(\widehat{B}_{i,j}(\omega_{i,j}))
\end{equation}
This time we consider two line segments. Firstly consider the segment
\begin{equation*}
    K:=\left\{K(t)=\bx + t \frac{\omega_{i,j}-\bz}{\|\omega_{i,j}-\bz\|_{2}}: 0\leq t \leq \|\omega_{i,j}-\bz\|_{2} \right\}. 
\end{equation*}
If $\omega_{i,j}=\bz$ simply take $K=K(0)=\{\bx\}$.
Note that $K$ is parallel to $\cH^{*}$ and $\|\omega_{i,j}-\bz\|_{2}\leq \frac{1}{4}r(\widehat{B}_{i,j}(\omega_{i,j}))$, so 
\begin{equation}\label{K path}
    K \subset \frac{1}{2}\widehat{B}_{i,j}(\omega_{i,j})\setminus \Delta(\cH^{*}, 9 \beta^{\frac{d}{d-1}}r(\widehat{B}_{i,j}(\omega_{i,j})) ).
\end{equation}
Now consider the second segment
\begin{equation*}
    M:=\left\{ M(t)=\bx+ \omega_{i,j}-\bz + t \frac{\hat{\bv}}{\sin \theta}: 0\leq t \leq \frac{1}{4}r(\widehat{B}_{i,j}(\omega_{i,j}))\right\}.
\end{equation*}
We know $\hat{\bv}$ is parallel to $\cH^{*}$. Observe that 
\begin{equation*}
    \|(\bx+ \omega_{i,j}-\bz) -\omega_{i,j}\|_{2}=10 \beta^{\tfrac{d}{d-1}}r(\widehat{B}_{i,j}(\omega_{i,j}))
\end{equation*}
and so, by \eqref{eq: beta parameter} and the triangle inequality 
\begin{equation}  \label{M path}
    M\subset \frac{1}{2}\widehat{B}_{i,j}(\omega_{i,j})\setminus \Delta(\cH^{*}, 9 \beta^{\frac{d}{d-1}}r(\widehat{B}_{i,j}(\omega_{i,j})) ).
\end{equation}
The path $K\cup M$ is connected, and by \eqref{M path} and \eqref{K path} we have 
\begin{equation*}
   K\cup  M\subset \frac{1}{2}\widehat{B}_{i,j}(\omega_{i,j})\setminus \Delta(\cH^{*}, 9 \beta^{\frac{d}{d-1}}r(\widehat{B}_{i,j}(\omega_{i,j})) ).
\end{equation*}
Consider the endpoint $M(t_{1})$ for
\begin{equation*}
    t_{1}=\frac{1}{4}r(\widehat{B}_{i,j}(\omega_{i,j})).
\end{equation*}
Note that \eqref{case2 theta large} implies $\cos\theta \leq 40 \beta^{\frac{d}{d-1}}$. Observe that
\begin{align*}
    \|\pi_{\widetilde{\cH}_{i}}(M(t_{1})-\omega_{i,j})\|_{2} &\leq \|\pi_{\widetilde{\cH}_{i}}(\bv_{*})\|_{2}10\beta^{\frac{d}{d-1}}r(\widehat{B}_{i,j}(\omega_{i,j})) + \tfrac{1}{4} r(\widehat{B}_{i,j}(\omega_{i,j}))\left\|\pi_{\widetilde{\cH}_{i}}\left(\tfrac{\hat{\bv}}{\|\hat{\bv}\|_{2}}\right)\right\|_{2}\\
    &\overset{\text{\eqref{eq: proj length star}+\eqref{eq: proj length hat}}}{\leq} 10\beta^{\frac{d}{d-1}}r(\widehat{B}_{i,j}(\omega_{i,j}))\sin \theta + \tfrac{1}{4} r(\widehat{B}_{i,j}(\omega_{i,j})) \cos\theta\\
    &\leq 20 \beta^{\frac{d}{d-1}} r(\widehat{B}_{i,j}(\omega_{i,j})).
\end{align*}
Also
\begin{align*}
    \langle M(t_{1})-\omega_{i,j}, \bv_{i}\rangle &= 10\beta^{\frac{d}{d-1}}r(\widehat{B}_{i,j}(\omega_{i,j})) \langle \bv_{*},\bv_{i}\rangle + \tfrac{1}{4}r(\widehat{B}_{i,j}(\omega_{i,j})) \left\langle \frac{\hat{\bv}}{\|\hat{\bv}\|_{2}}, \bv_{i}\right\rangle\\
    &\overset{\text{\eqref{eq: hat height}}}{=}10\beta^{\frac{d}{d-1}}r(\widehat{B}_{i,j}(\omega_{i,j})) \cos \theta + \tfrac{1}{4}r(\widehat{B}_{i,j}(\omega_{i,j})) \sin \theta\\
    &\geq \tfrac{1}{4}r(\widehat{B}_{i,j}(\omega_{i,j})) \sin \theta.
\end{align*}
By \eqref{case2 theta large} $\sin\theta> (1+40^2\beta^{2\frac{d}{d-1}})^{-\frac{1}{2}}$, so by \eqref{eq: beta parameter},
\begin{equation*}
    \langle M(t_{1})-\omega_{i,j}, \bv_{i}\rangle \geq \tfrac{1}{5}r(\widehat{B}_{i,j}(\omega_{i,j})).
\end{equation*}
Then
\begin{equation*}
    \langle\Phi(\pi_{\widetilde{\cH}_{i}}(M(t_{1})))-\omega_{i,j}, \bv_{i} \rangle\overset{\text{\eqref{eq: s height}}}{\leq} 41 \beta^{\frac{1}{d-1}}r(\widehat{B}_{i,j}(\omega_{i,j}))\overset{\text{\eqref{eq: beta parameter}}}{<} \tfrac{1}{5}r(\widehat{B}_{i,j}(\omega_{i,j}))\leq \langle M(t_{1})-\omega_{i,j}, \bv_{i}\rangle,
\end{equation*}
and so $M(t_{1})$ lies above $S_{i}$, and $K(0)=\bx$ lies below $S_{i}$. Therefore the path $K\cup M$ satisfies the properties of that $P$, and so we are done for the case \eqref{case2 theta large}. This completes all cases and so the claim is proven.
\end{proof}
Consider the cylinder
\begin{equation*}
   C_{i,j}:= C_{i,j}(\omega_{i,j},\tfrac{\bp_{i,j}}{q_{i,j}}):=\bigcup_{0\leq \eta\leq 1} B\left(\eta \omega_{i,j} +(1-\eta)\tfrac{\bp_{i,j}}{q_{i,j}}, \tfrac{1}{2}r(\widehat{B}_{i,j}(\omega_{i,j})) \right).
\end{equation*}
Observe the cylinder has the following key properties:
\begin{enumerate}[(A)]
\item \label{P1} We have $\omega_{i,j} \in S_{i}$, $\tfrac{\bp_{i,j}}{q_{i,j}}\in \cQ_{i}(\omega_{i,j})$, and 
    \begin{equation*}
        \tfrac{1}{2}\widehat{B}_{i,j}(\omega_{i,j}) \subset (1-\beta)\widehat{B}_{i,j}(\omega_{i,j}).
    \end{equation*}
    
    \item\label{P2} We have $ \|\omega_{i,j}-\tfrac{\bp_{i,j}}{q_{i,j}}\|_{2}=\varepsilon q_{i,j}^{-\frac{d+1}{d}}$

    \item \label{P3} The angle $\alpha$ between the normal vector of the affine hyperplane $\cH^{*}$ and the line containing the segment $\{\eta \omega_{i,j} +(1-\eta)\tfrac{\bp_{i,j}}{q_{i,j}}: 0\leq \eta\leq 1\}$ satisfies 
    \begin{equation*}
        \cos \alpha \geq \frac{1-40\beta^{\frac{d}{d-1}}}{1+40\beta^{\frac{d}{d-1}}}.
    \end{equation*}
    This follows from \eqref{eq: angle target}.  
\end{enumerate}
Using that 
\begin{equation*}
    \left\|\omega_{i,j}-\tfrac{\bp_{i,j}}{q_{i,j}}\right\|_{2}>r(\widehat{B}_{i,j}(\omega_{i,j})),
\end{equation*}
We calculate
\begin{align}
    \lambda_{d}(C_{i,j})&\leq 2 \lambda_{d-1}(B(0,1)) r(\widehat{B}_{i,j}(\omega_{i,j}))^{d-1}\left\|\omega_{i,j}-\tfrac{\bp_{i,j}}{q_{i,j}}\right\|_{2} \nonumber\\%+\lambda_{d}(B(0,1))\tfrac{1}{2^{d}}r(\widehat{B}_{i,1}(\omega_{i,1}))^{d}\nonumber\\
    & \overset{\text{\eqref{P1}+\eqref{P2}}}{\leq} 2 \lambda_{d-1}(B(0,1))\left(\beta^{k_{i}+\ell_{i}+1+j\tfrac{d}{d-1}}r(B_{N_{i}}) \right)^{d-1}\varepsilon \widetilde{Q}_{i,k_{i}-1}^{-\tfrac{d+1}{d}} \nonumber\\
    & \overset{\text{\eqref{eq: ell size}}}{\leq} 2 \lambda_{d-1}(B(0,1))\left(\beta^{k_{i}+\ell_{i}+1+j\tfrac{d}{d-1}}r(B_{N_{i}}) \right)^{d-1}\beta^{k_{i}+\ell_{i}-1}r(B_{N_{i}}) \nonumber\\
    & \leq 2 \lambda_{d-1}(B(0,1)) \beta^{d-2} \left(\beta^{k_{i}+\ell_{i}+j}r(B_{N_{i}})\right)^{d}\nonumber\\
    &\leq 2^{2-d} \tfrac{\lambda_{d-1}(B(0,1))}{\lambda_{d}(B(0,1))} \beta^{d-2} \lambda_{d}\left(2\widetilde{B}_{i,k_{i}+\ell_{i}+j}\right).\label{eq: cylinder measure}
\end{align}
Applying the Simplex lemma to $C_{i,j}$ we see that all rational points $\tfrac{\bp'}{q'}\in \Q^{d}\cap C_{i,j}$ with
\begin{equation*}
    1\leq q'< (d!)^{-\tfrac{1}{d+1}} \lambda_{d}(C_{i,j})^{-\tfrac{1}{d+1}} 
\end{equation*}
lie on a rational affine hyperplane, say $\cH_{M}$. Note by \eqref{eq: cylinder measure} that
\begin{align} \label{eq: cylinder small enough}
    (d!)^{-\tfrac{1}{d+1}} \lambda_{d}(C_{i,j})^{-\tfrac{1}{d+1}} &\geq \left(\tfrac{\beta}{2}\right)^{\tfrac{2-d}{d+1}}\left(\tfrac{\lambda_{d}(B(0,1))}{\lambda_{d-1}(B(0,1))}\right)^{\tfrac{1}{d+1}} \widetilde{Q}_{i,k_{i}+\ell_{i}+j}, \nonumber\\
    & \overset{\text{\eqref{eq: beta parameter}}}{\geq} \widetilde{Q}_{i,k_{i}+\ell_{i}+j}.
\end{align}
Furthermore, $\frac{\bp_{i,j}}{q_{i,j}}\in C_{i,j}\cap \cQ_{i}$, and since $2\widetilde{Q}_{i,k_{i}}<\widetilde{Q}_{i,k_{i}+\ell_{i}+j}$ the Simplex lemma tells us $\tfrac{\bp_{i,j}}{q_{i,j}} \in \cH_{M}$. 

Let $\theta$ be the angle between $\cH_{M}$ and $\cH^{*}$, $\alpha$ be the angle between the line described by the pair of points $(\omega_{i,j},\tfrac{\bp_{i,j}}{q_{i,j}})$ and a vector normal to $\cH^{*}$, and $\gamma$ be the angle between the line described described by the pair of points $(\omega_{i,j},\tfrac{\bp_{i,j}}{q_{i,j}})$ and a normal vector to $\cH_{M}$. By \eqref{P3} we have that 
\begin{equation*}
    \cos \alpha\geq 1-\frac{80\beta^{\frac{d}{d-1}}}{1+40\beta^{\frac{d}{d-1}}}
\end{equation*}
Suppose that $\cH_{M}\cap B(\omega_{i,j}, \tfrac{1}{2}r(\widehat{B}_{i,j}(\omega_{i,j}))) = \emptyset$. Then choose the ball
\begin{equation*}
    B(\omega_{i,j}, \beta^{k_{i}+\ell_{i}+1+(j+1)\frac{d}{d-1}}r(B_{N_{i}}))\subset \tfrac{1}{4}\widehat{B}_{i,j}(\omega_{i,j}).
\end{equation*}
Since
\begin{equation*}
    18 \beta^{k_{i}+\ell_{i}+1+(j+1)\frac{d}{d-1}}r(B_{N_{i}})<\tfrac{1}{4}r(\widehat{B}_{i,j}(\omega_{i,j})
\end{equation*}
it follows that
\begin{equation*}
    B\left(\omega_{i,j}, \beta^{k_{i}+\ell_{i}+1+(j+1)\frac{d}{d-1}}r(B_{N_{i}})\right)\subset \tfrac{1}{4}\widehat{B}_{i,j}(\omega_{i,j})\setminus \Delta(\cH_{M}, 18 \beta^{k_{i}+\ell_{i}+1+(j+1)\frac{d}{d-1}}r(B_{N_{i}})).
\end{equation*}

So assume $\cH_{M}\cap B(\omega_{i,j}, \tfrac{1}{2}r(\widehat{B}_{i,j}(\omega_{i,j}))) \neq \emptyset$. Then it can be shown that
\begin{equation*}
    \sin \gamma\geq \left(1-\frac{1}{4}\left(\frac{r(\widehat{B}_{i,j}(\omega_{i,j}))}{\left\|\omega_{i,j}-\tfrac{\bp_{i,j}}{q_{i,j}}\right\|_{2}} \right)^{2}\right)^{\tfrac{1}{2}}.
\end{equation*}
By \eqref{P1} and \eqref{P2} we know 
\begin{equation*}
    \left\|\omega_{i,j}-\tfrac{\bp_{i,j}}{q_{i,j}}\right\|_{2}\geq 2^{-1-\frac{d+1}{d}}\varepsilon \widetilde{Q}_{i,k_{i}}^{-\tfrac{d+1}{d}}>2^{-1-\frac{d+1}{d}}\beta^{k_{i}+\ell_{i}+2}r(B_{N_{i}})
\end{equation*}
and so
\begin{equation} \label{eq: theta angle}
    \sin \gamma \geq \left(1-2^{2\frac{d+1}{d}}\beta^{2\left(j\tfrac{d}{d-1}-1\right)}\right)^{\tfrac{1}{2}}.
\end{equation}
Note that by \eqref{eq: beta parameter} we can see $\gamma>\alpha$ and so $\theta\geq \gamma-\alpha$. Hence we have that
\begin{align*}
    \sin \theta &\geq  \sin(\gamma-\alpha) \\
    &\geq \sin\gamma\cos\alpha -\sin\alpha\\
    & \overset{\text{\eqref{eq: beta parameter}+\eqref{eq: theta angle}+\eqref{P3}}}{\geq} \frac{1}{2}
\end{align*}
 For $\beta>0$ satisfying \eqref{eq: beta parameter} we have that
\begin{equation*}
    \frac{1}{2}> \frac{144 \beta^{\tfrac{d}{d-1}}}{1-72 \beta^{\tfrac{d}{d-1}}}
\end{equation*}
So we can apply Lemma~\ref{lem: colliding hyperplanes} to $\cH=\cH^{*}$, $\cH'=\cH_{M}$, $B(\bx,r)=\tfrac{1}{4}\widehat{B}_{i,j}(\omega_{i,j})$ and $C=72\beta^{\tfrac{d}{d-1}}$. Hence there exists some $\bz \in \cH^{*}$ such that
\begin{equation} \label{eq: avoid rationals case 2}
    B\left(\bz, 18\beta^{k_{i}+\ell_{i}+1+(j+1)\tfrac{d}{d-1}}r(B_{N_{i}}) \right) \subset \tfrac{1}{4}\widehat{B}_{i,j}(\omega_{i,j}) \setminus \Delta\left(\cH_{M}, 18\beta^{k_{i}+\ell_{i}+1+(j+1)\tfrac{d}{d-1}}r(B_{N_{i}}) \right).
\end{equation}
Note that
\begin{align*}
    \beta^{k_{i}+\ell_{i}+1+(j+1)\tfrac{d}{d-1}}r(B_{N_{i}})&\geq \beta^{2+(j+1)\tfrac{d}{d-1}}\varepsilon \widetilde{Q}_{i,k_{i}-1}^{-\tfrac{d+1}{d}}\nonumber \\
    &= \beta^{2+(j+1)\tfrac{d}{d-1}-\ell_{i}-j}\varepsilon \widetilde{Q}_{i,k_{i}+\ell_{i}+j-1}^{-\tfrac{d+1}{d}} \nonumber \\
    &\overset{\text{\eqref{eq: t size}}}{>}\varepsilon \widetilde{Q}_{i,k_{i}+\ell_{i}+j-1}^{-\tfrac{d+1}{d}}
\end{align*}
provided $j+1\leq t_{i}$, and so, by \eqref{eq: cylinder small enough} and since
\begin{equation} \label{eq: hat larger than tilde}
    \tfrac{1}{4}r(\widehat{B}_{i,j}(\omega_{i,j}))>\varepsilon \widetilde{Q}_{i,k_{i}+\ell_{i}+j-1}^{-\tfrac{d+1}{d}},
\end{equation}
we have
\begin{align}
    \Delta\left(\cH_{M}, 18\beta^{k_{i}+\ell_{i}+1+(j+1)\tfrac{d}{d-1}}r(B_{N_{i}}) \right) &\supseteq \Delta\left(\cH_{M}, \varepsilon \widetilde{Q}_{i,k_{i}+\ell_{i}+j-1}^{-\tfrac{d+1}{d}} \right) \nonumber\\
    &\hspace{-2cm} \supseteq B\left(\bx, \tfrac{1}{4}r(\widehat{B}_{i,j}(\omega_{i,j})) \right) \cap \bigcup_{\widetilde{Q}_{i,k_{i}+\ell_{i}+j-1}\leq q < \widetilde{Q}_{i,k_{i}+\ell_{i}+j}}\bigcup_{\bp\in \Z^{d}} B\left( \tfrac{\bp}{q}, \varepsilon q^{-\tfrac{d+1}{d}} \right) \label{eq: rationals we avoid case 2}
\end{align}
Consider the two point 
\begin{equation*}
    \bz_{-}:=\bz-9\beta^{k_{i}+\ell_{i}+1+(j+1)\frac{d}{d-1}}r(B_{N_{i}})\bv_{*}, \qquad \bz_{+}:=\bz+9\beta^{k_{i}+\ell_{i}+1+(j+1)\frac{d}{d-1}}r(B_{N_{i}})\bv_{*},
\end{equation*}
and the points $\Phi(\pi_{\widetilde{\cH}_{i}}(\bz_{-}))$, $\Phi(\pi_{\widetilde{\cH}_{i}}(\bz_{+}))$. By \eqref{eq: contained in S} and \eqref{lem: containment condition} these points are defined and belong to $S_{i}$. 
By \eqref{eq: hyperplane covering surface} we must have that
\begin{equation*}
    \langle \bz_{+}, \bv_{i} \rangle \geq \langle \Phi(\pi_{\widetilde{\cH}_{i}}(\bz_{+})), \bv_{i}\rangle, \quad \text{ and } \quad \langle \bz_{-}, \bv_{i} \rangle \leq \langle \Phi(\pi_{\widetilde{\cH}_{i}}(\bz_{-})), \bv_{i}\rangle.
\end{equation*}
Therefore, by considering the endpoints of the segment
\begin{equation*}
    \left\{\bz_{-}+t \bv_{*}: 0\leq t \leq 18\beta^{k_{i}+\ell_{i}+1+(j+1)\frac{d}{d-1}}r(B_{N_{i}})\right\}
\end{equation*}
and seeing that one point lies above $S_{i}$ and one point lies below $S_{i}$, by $S_{i}$ being a continuous $(d-1)$-dimensional surface we have that
\begin{equation*}
    \left\{\bz_{-}+t \bv_{*}: 0\leq t \leq 18\beta^{k_{i}+\ell_{i}+1+(j+1)\frac{d}{d-1}}r(B_{N_{i}})\right\} \cap S_{i}\neq \emptyset. 
\end{equation*} 
and so we can choose a point $\omega_{i,j+1} \in S_{i}\cap B\left(\bz, 9\beta^{k_{i}+\ell_{i}+1+(j+1)\tfrac{d}{d-1}}r(B_{N_{i}}) \right)$ such that 
\begin{equation*}
    \widehat{B}_{i,j+1}(\omega_{i,j+1}) \subset B\left(\bz, 18\beta^{k_{i}+\ell_{i}+1+(j+1)\tfrac{d}{d-1}}r(B_{N_{i}}) \right)\overset{\text{\eqref{eq: avoid rationals case 2}}}{\subset} \tfrac{1}{4}\widehat{B}_{i,j}(\omega_{i,j}) \overset{\text{\eqref{P1}}}{\subset} (1-\beta)\widehat{B}_{i,j}(\omega_{i,j}).
\end{equation*}
Hence \eqref{eq: nested iterative assumption} is satisfied.
By \eqref{eq: avoid rationals case 2}, \eqref{eq: rationals we avoid case 2}, and \eqref{eq: hat larger than tilde}
we have
\begin{equation} \label{eq: A}
    \widehat{B}_{i,j+1}(\omega_{i,j+1}) \cap \bigcup_{\widetilde{Q}_{i,k_{i}+\ell_{i}+j-1}\leq q < \widetilde{Q}_{i,k_{i}+\ell_{i}+j}}\bigcup_{\bp\in \Z^{d}} B\left( \tfrac{\bp}{q}, \varepsilon q^{-\tfrac{d+1}{d}} \right) = \emptyset.
\end{equation}
Combining \eqref{eq: A} with \eqref{eq: avoid rational iterative assumption} we obtain
\begin{equation*}
    \widehat{B}_{i,j+1}(\omega_{i,j+1})\cap S_{i}^{+} \cap \bigcup_{Q_{N_{i-1}}\leq q < \widetilde{Q}_{i,k_{i}+\ell_{i}+j}}\bigcup_{\bp\in \Z^{d}}B\left( \tfrac{\bp}{q},\varepsilon q^{-\tfrac{d+1}{d}}\right)=\emptyset
\end{equation*}
completing the inductive step in this subcase.
\end{itemize}
This completes all possible cases and so the proof of the Lemma~\ref{lem: spyglass lem} is complete.
\end{proof}

\subsubsection{Completing the exceptional turn}
To finish the exceptional turn take the ball $\widehat{B}_{i,t_{i}}(\omega_{i,t_{i}})$ as given in Lemma~\ref{lem: spyglass lem}. We want to pick a ball $\check{B}_{i,1}$ such that
\begin{align} \label{eq: check conditions 1}
r(\check{B}_{i,1})&=\frac{1}{9}\beta r(\widehat{B}_{i,t_{i}}(\omega_{i,t_{i}})),\\
    \check{B}_{i,1}&\subset \widehat{B}_{i,t_{i}}(\omega_{i,t_{i}})\cap S_{i}^{+}.\label{eq: check conditions 2}\\
    \check{B}_{i,1}& \cap \bigcup_{\widetilde{Q}_{i,k_{i}+\ell_{i}+t_{i}-1}\leq q < \widetilde{Q}_{i,k_{i}+\ell_{i}+1+t_{i}\frac{d}{d-1}}} \bigcup_{\bp \in \Z^{d}} B\left(\tfrac{\bp}{q}, \varepsilon q^{-\tfrac{d+1}{d}} \right)=\emptyset. \label{eq: check conditions 3}
\end{align}
Note we are no longer choosing our centre point in $S_{i}$. Take $\omega_{i,t_{i}}$, take $\bv_{i}$ (recall this is the unit vector normal to the affine hyperplane $\widetilde{\cH}_{i}$, which is the hyperplane on which $S_{i}$ is built around) and consider the bounded cone
\begin{align*}
    D_{\omega_{i,t_{i}}}&:=\left\{ \omega_{i,t_{i}}+\bu : 0<\|\bu\|_{2}\leq r(\widehat{B}_{i,t_{i}}(\omega_{i,t_{i}})) \quad \text{ and } \tfrac{\bu}{\|\bu\|_{2}} \cdot \bv_{i} \geq \left(1-\tfrac{1}{4}\beta^{2}\right)^{\tfrac{1}{2}}  \right\} \\
    &\subset \widehat{B}_{i,t_{i}}(\omega_{i,t_{i}})\cap S_{i}^{+}.
\end{align*}
This follows from \eqref{eq: cone of avoidence}. Take the ball
\begin{equation*}
    \check{B}:=B\left( \omega_{i,t_{i}}+(1-\tfrac{1}{3}\beta)r(\widehat{B}_{i,t_{i}}(\omega_{i,t_{i}}))\bv_{i},\tfrac{1}{3}\beta r(\widehat{B}_{i,t_{i}}(\omega_{i,t_{i}})) \right).
\end{equation*}
Letting $\alpha$ denote the half-angle of the cone $D_{\omega_{i,t_{i}}}$, it can be seen that
\begin{equation*}
    (1-\tfrac{1}{3}\beta)r(\widehat{B}_{i,t_{i}}(\omega_{i,t_{i}})) \sin \alpha \geq (1-\tfrac{1}{3}\beta)\tfrac{1}{2}\beta r(\widehat{B}_{i,t_{i}}(\omega_{i,t_{i}}))\overset{\text{\eqref{eq: beta parameter}}}{>}\tfrac{1}{3}\beta r(\widehat{B}_{i,t_{i}}(\omega_{i,t_{i}})).
\end{equation*}
 Hence $\check{B} \subset D_{\omega_{i,t_{i}}}$, and so picking any ball inside $\check{B}$ will ensure \eqref{eq: check conditions 2} is satisfied.

Apply the Simplex Lemma to $2\widehat{B}_{i,t_{i}}(\omega_{i,t_{i}})$ to see that all rational points $\tfrac{\bp}{q}\in \Q^{d}\cap 2\widehat{B}_{i,t_{i}}(\omega_{i,t_{i}})$ with
\begin{equation*}
    1 \leq q < (d!)^{-\tfrac{1}{d+1}}\lambda_{d}(2\widehat{B}_{i,t_{i}}(\omega_{i,t_{i}}))^{-\tfrac{1}{d+1}}=\widetilde{Q}_{i,k_{i}+\ell_{i}+1+t_{i}\frac{d}{d-1}}
\end{equation*}
lie on a rational affine hyperplane, say $\check{\cH}_{i,1}$. Using that
\begin{align*}
     \varepsilon \widetilde{Q}_{i, k_{i}+\ell_{i}+t_{i}-1}^{-\tfrac{d+1}{d}} &\overset{\text{\eqref{eq: ell size}}}{\leq} \beta^{k_{i}+\ell_{i}+\ell_{i}+t_{i}}r(B_{N_{i}})\\
     & \overset{\text{\eqref{eq: t size}}}{=} \beta^{k_{i}+\ell_{i}(d+1)-4(d-1)}r(B_{N_{i}})\\
     & = \beta^{3}\beta^{k_{i}+\ell_{i}(d+1)+1-4d}r(B_{N_{i}})\\
     & = \beta^{3} r(\widehat{B}_{i,t_{i}}(\omega_{i,t_{i}}))\\
     &\overset{\text{\eqref{eq: beta parameter}}}{\leq}\tfrac{1}{9} \beta r(\widehat{B}_{i,t_{i}}(\omega_{i,t_{i}})).
\end{align*}
we have
\begin{equation} \label{eq: avoiding rationals AGAIN}
    \Delta(\check{\cH}_{i,1}, \varepsilon \widetilde{Q}_{i, k_{i}+\ell_{i}+t_{i}-1}^{-\tfrac{d+1}{d}}) \supseteq \widehat{B}_{i,t_{i}}(\omega_{i,t_{i}}) \cap \bigcup_{\widetilde{Q}_{i,k_{i}+\ell_{i}+t_{i}-1}\leq q < \widetilde{Q}_{i,k_{i}+\ell_{i}+1+t_{i}\frac{d}{d-1}}} \bigcup_{\bp \in \Z^{d}} B\left(\tfrac{\bp}{q}, \varepsilon q^{-\tfrac{d+1}{d}} \right).
\end{equation}

Thus, there exists ball $\check{B}_{i,1}$ with $r(\check{B}_{i,1})=\frac{1}{9}\beta r(\widehat{B}_{i,t_{i}}(\omega_{i,t_{i}}))$, such that
\begin{equation*}
    \check{B}_{i,1} \subset \check{B} \setminus \Delta(\check{\cH}_{i,1}, \varepsilon \widetilde{Q}_{i, k_{i}+\ell_{i}+t_{i}-1}^{-\frac{d+1}{d}}).
\end{equation*}
and so by \eqref{eq: avoiding rationals AGAIN} $\check{B}_{i,1}$ satisfies \eqref{eq: check conditions 1} and \eqref{eq: check conditions 3}. 
Observe that 
\begin{equation*}
    r(\check{B}_{i,1})\overset{\text{\eqref{eq: check conditions 1}}}{=} \frac{1}{9}\beta^{k_{i}+\ell_{i}+1+t_{i}\frac{d}{d-1}+1}r(B_{N_{i}})\overset{\text{\eqref{eq: t size}}}{=}\frac{1}{9}\beta^{k_{i}+\ell_{i}(d+1)+2-4d}r(B_{N_{i}})
\end{equation*}
Applying the Simplex Lemma to $2\check{B}_{i,1}$ we see that all rational points $\tfrac{\bp}{q}\in \Q^{d}\cap 2\check{B}_{i,1}$ with
\begin{equation*}
    1\leq q < (d!)^{-\tfrac{1}{d+1}} \lambda_{d}(2\check{B}_{i,1})^{-\tfrac{1}{d+1}}=9^{\frac{d}{d+1}}\widetilde{Q}_{i,k_{i}+\ell_{i}(d+1)+2-4d}
\end{equation*}
lie on an affine hyperplane, say $\check{\cH}_{i,2}$. Note that
\begin{align*}
    \Delta(\check{\cH}_{i,2}, \beta r(\check{B}_{i,1}))&\supset \Delta(\check{\cH}_{i,2}, \varepsilon \widetilde{Q}_{i,k_{i}+\ell_{i}(d+1)+1-4d}^{-\tfrac{d+1}{d}})\\
    &\supset \check{B}_{i,1}\cap \bigcup_{\widetilde{Q}_{i,k_{i}+\ell_{i}(d+1)+1-4d} \leq q< \widetilde{Q}_{i,k_{i}+\ell_{i}(d+1)+2-4d}} \bigcup_{\bp \in \Z^{d}} B\left( \tfrac{\bp}{q}, \varepsilon q^{-\tfrac{d+1}{d}}\right)
\end{align*}
since
\begin{equation*}
    \varepsilon \widetilde{Q}_{i,k_{i}+\ell_{i}(d+1)+1-4d}^{-\tfrac{d+1}{d}} \overset{\text{\eqref{eq: ell size}}}{\leq} \beta^{k_{i}+\ell_{i}(d+1)+\ell_{i}-4d}r(B_{N_{i}})\leq \beta r(\check{B}_{i,1}),
\end{equation*}
since $\ell_{i}>4$ and \eqref{eq: beta parameter}. Therefore, by \eqref{eq: beta parameter} we may pick a ball $\check{B}_{i,2}$ such that
\begin{equation}\label{eq: rational miss 1}
    \check{B}_{i,2} \subset \check{B}_{i,1}\setminus \Delta(\check{\cH}_{i,2}, \beta r(\check{B}_{i,1})) \quad \text{ and } \quad r(\check{B}_{i,2})=9\beta r(\check{B}_{i,1}).
\end{equation}
Repeating this process, that is from $\check{B}_{i,j}$ for some $j\geq 2$, we can find a ball
\begin{align*}
    \check{B}_{i,j+1}& \subset \check{B}_{i,j}\qquad \text{ with } r(\check{B}_{i,j+1})=\beta r(\check{B}_{i,j}) \\
    \check{B}_{i,j+1}& \cap \bigcup_{\widetilde{Q}_{i,k_{i}+\ell_{i}(d+1)+j-4d}\leq q < \widetilde{Q}_{i,k_{i}+\ell_{i}(d+1)+j+1-4d}}\bigcup_{\bp\in \Z^{d}}B\left( \frac{\bp}{q}, \varepsilon q^{-\frac{d+1}{d}}\right)=\emptyset.
\end{align*}
Note the first iteration is exceptional in the multiplicative constant $9$ appearing in the radius bound, later iterations do not include the constant $9$. Repeat up to some $u_{i}\in \N$, arriving at a ball $\check{B}_{i,u_{i}}$ such that
\begin{align*}
    r(\check{B}_{i,u_{i}})&=\beta^{R_{i}}r(B_{N_{i}}) \qquad \text{(Choose $u_{i}$ such that $u_{i}+k_{i}+\ell_{i}(d+1)+1-4d=R_{i}$)},\\
    \check{B}_{i,u_{i}}&\subset \check{B}_{i,1} \overset{\text{\eqref{eq: check conditions 2}}}{\subset} \widehat{B}_{i,t_{i}}(\omega_{i,t_{i}}) \overset{\text{\eqref{lem: containment condition}}}{\subset} \widehat{B}_{i,1}(\omega_{i,1}) \overset{\text{\eqref{eq: Bhat 1 contained in Btilde}}}{\subset} \widetilde{B}_{i,k_{i}+\ell_{i}} \overset{\text{\eqref{eq: shrink still contains ball}}}{\subset} \widetilde{B}_{i,k_{i}} \overset{\text{\eqref{eq:contained in big ball}}}{\subset} B_{N_{i}},
\end{align*}
and
\begin{equation} \label{eq: last stages of avoiding rationals}
    \check{B}_{i,u_{i}}\cap \bigcup_{\widetilde{Q}_{i,k_{i}+\ell_{i}(d+1)+1-4d}\leq q < \widetilde{Q}_{i,R_{i}-1}} \bigcup_{\bp\in \Z^{d}}B\left(\tfrac{\bp}{q}, \varepsilon q^{-\frac{d+1}{d}}\right)=\emptyset.
\end{equation}
Note our choice of $u_{i}$ can be seen to be positive by \eqref{eq: ell size explicit}, $0\leq k_{i} \leq 2(d+1)$, and \eqref{eq: Ri parameter}.
Alice legally picks $A_{N_{i}}=\check{B}_{i,u_{i}}$ and so Bobs $N_{i}+1$th ball is $B_{N_{i}+1}=A_{N_{i}}$. Observe that
\begin{align} \label{eq: exceptional turn avoid rationals a}
    B_{N_{i}+1} \cap \underset{\widetilde{Q}_{i,k_{i}+\ell_{i}(d+1)+2-4d}\leq q < \widetilde{Q}_{i,R_{i}-1} \eqref{eq: last stages of avoiding rationals} }{\underset{\widetilde{Q}_{i,k_{i}+\ell_{i}+1+t_{i}\frac{d}{d-1}} \leq q< \widetilde{Q}_{i,k_{i}+\ell_{i}(d+1)+2-4d} \eqref{eq: rational miss 1}}{\underset{\widetilde{Q}_{i,k_{i}+\ell_{i}+t_{i}-1}\leq q <\widetilde{Q}_{i,k_{i}+\ell_{i}+1+t_{i}\frac{d}{d-1}}\eqref{eq: check conditions 3}}{\underset{Q_{N_{i-1}}\leq q<\widetilde{Q}_{i,k_{i}+\ell_{i}+t_{i}-1}\eqref{lem: avoid rationals condition}}{\bigcup}}}} \bigcup_{\bp \in \Z^{d}} B\left(\tfrac{\bp}{q}, \varepsilon q^{-\tfrac{d+1}{d}} \right)=\emptyset.
\end{align}
Moreover, since $B_{N_{i}+1}\subset \widehat{B}_{i,t_{i}}(\omega_{i,t_{i}})$ for some $\omega_{i,t_{i}}\in S_{i}$, so $$\|\omega_{i,t_{i}}-\tfrac{\bp_{i,t_{i}}}{q_{i,t_{i}}}\|_{2}=\varepsilon q_{i,t_{i}}^{-\tfrac{d+1}{d}} \quad \text{for some} \quad \tfrac{\bp_{i,t_{i}}}{q_{i,t_i}}\in \cQ_{i},$$
we have that
\begin{equation*}
B_{N_{i}+1} \subset B\left(\tfrac{\bp_{i,t_{i}}}{q_{i,t_{i}}}, \varepsilon q_{i,t_{i}}^{-\tfrac{d+1}{d}}+r(\widehat{B}_{i,t_{i}}(\omega_{i,t_{i}})) \right) \setminus B\left(\tfrac{\bp_{i,t_{i}}}{q_{i,t_{i}}}, \varepsilon q_{i,t_{i}}^{-\tfrac{d+1}{d}} \right)
\end{equation*}
with
\begin{align*}
    r(\widehat{B}_{i,t_{i}}(\omega_{i,t_{i}})) &=\beta^{k_{i}+\ell_{i}+1+t_{i}\tfrac{d}{d-1}}r(B_{N_{i}})\overset{\text{\eqref{eq: ell size}}}{\leq} \beta^{\ell_{i}d -4d-2} \varepsilon \widetilde{Q}_{i,k_{i}+1}^{-\tfrac{d+1}{d}} \leq \beta^{\ell_{i}d-4d-2} \varepsilon q_{i,t_{i}}^{-\tfrac{d+1}{d}}\\
    &\leq \beta^{-6d-2}2^{d}(d!\lambda_{d}(B(0,1))) \varepsilon^{d+1} q_{i,t_{i}}^{-\tfrac{d+1}{d}},
\end{align*}
using
\begin{equation*}
    \widetilde{Q}_{i,k_{i}-1}^{-\tfrac{d+1}{d}}=2(d!\lambda_{d}(B(0,1)))^{\frac{1}{d}}\beta^{k_{i}-1}r(B_{N_{i}}).
\end{equation*}
and \eqref{eq: ell size} to see that
\begin{equation*}
    \beta^{\ell_{i}+2} \leq 2(d!\lambda_{d}(B(0,1)))^{\frac{1}{d}}\varepsilon.
\end{equation*}
Therefore
\begin{equation} \label{eq: exceptional turn close to rational}
    B_{N_{i}+1}\subset B\left(\tfrac{\bp_{i,t_{i}}}{q_{i,t_{i}}}, \left(1+\beta^{-6d-2}2^{d}(d!\lambda_{d}(B(0,1))) \varepsilon^{d}\right) \varepsilon q_{i,t_{i}}^{-\tfrac{d+1}{d}} \right) \setminus B\left(\tfrac{\bp_{i,t_{i}}}{q_{i,t_{i}}}, \varepsilon q_{i,t_{i}}^{-\tfrac{d+1}{d}} \right)
\end{equation}

We now show that our assumptions \eqref{eq: assumption 1} and \eqref{eq: assumption 2} at the start of the exceptional turn were valid. Note that \eqref{eq: assumption 1} is satisfied by Alice strategy in her regular turns. It remains to show \eqref{eq: assumption 2} is satisfied. This is achieved by noting that \eqref{eq: exceptional turn avoid rationals a} shows
\begin{equation} \label{eq: exceptional turn avoid rationals}
    B_{N_{i}+1}\cap \bigcup_{Q_{N_{i-1}}\leq q<\widetilde{Q}_{i,R-1}} \bigcup_{\bp\in\Z^{d}}B\left(\tfrac{\bp}{q}, \varepsilon q^{-\frac{d+1}{d}}\right)=\emptyset.
\end{equation}
For $1\leq q < Q_{N_{i-1}}$, using \eqref{eq: assumption 1} and \eqref{eq: assumption 2}, and for $Q_{N_{i-1}}\leq q< \widetilde{Q}_{i,R_{i}-1}$, using the monotonicity of $q\mapsto \varepsilon q^{-\frac{d+1}{d}}$ it can be seen that
\begin{equation*}
    \delta_{i}=\min\left\{ \delta_{i-1}, r(B_{0}), \varepsilon \widetilde{Q}_{i,R_{i}-1}^{-\tfrac{d+1}{d}}\right\}
\end{equation*}
is a suitable choice. Note $\delta_{i}$ is independent of the choice of $N_{i+1}$.
By the legality of Bobs moves we will have that $B_{N_{i+1}}\subset B_{N_{i}+1}$, and so
\begin{equation*}
    B_{N_{i+1}}\cap \bigcup_{1\leq q<\widetilde{Q}_{i,R_i -1}} \bigcup_{\bp\in\Z^{d}}B\left(\tfrac{\bp}{q}, \delta_{i}\right)=\emptyset,
\end{equation*}
thus \eqref{eq: assumption 2} is satisfied.

\subsection{The desired outcome}

It remains to show that the outcome of the game, $\bx_{\infty}$ satisfies 
\begin{equation*}
    \bx_{\infty} \in \left\{ \bx \in \R^{d}: \varepsilon \leq \Theta_{d}(\bx) \leq  \varepsilon\left(1+\beta^{-6d-2}2^{d}d!\lambda_{d}(B(0,1)) \varepsilon^{d}\right) \right\}.
\end{equation*}
Observe that \eqref{eq: regular turn avoid rationals} and \eqref{eq: exceptional turn avoid rationals} imply that at every turn, whether regular or exceptional, we have
\begin{equation*}
    B_{n+1} \cap \bigcup_{Q_{n-1}\leq q< Q_{n}}\bigcup_{\bp\in \Z^{d}}B\left( \tfrac{\bp}{q}, \varepsilon q^{-\tfrac{d+1}{d}}\right)=\emptyset,
\end{equation*}
and since $\bx_{\infty} \in \lim_{n \to \infty} \bigcap_{i\leq n} B_{i}$ we have that $\Theta_{d}(\bx_{\infty})\geq \varepsilon$. By \eqref{eq: exceptional turn close to rational} for every $\frac{\bp}{q}\in\left(\frac{\bp_{i,t_{i}}}{q_{i,t_{i}}}\right)_{i\in \N}$ we have 
\begin{equation*}
    \bx_{\infty} \in B\left(\tfrac{\bp}{q}, \left(1+\beta^{-6d-2}2^{d}(d!\lambda_{d}(B(0,1))) \varepsilon^{d}\right) \varepsilon q^{-\tfrac{d+1}{d}} \right) \setminus B\left(\tfrac{\bp}{q}, \varepsilon q^{-\tfrac{d+1}{d}} \right)
\end{equation*}
and so
\begin{equation*}
    \Theta_{d}(\bx_{\infty})\leq \varepsilon\left(1+\beta^{-6d-2}2^{d}(d!\lambda_{d}(B(0,1))) \varepsilon^{d}\right).
\end{equation*}
Hence the proof is complete.

\section{Further remarks}
As a further remark we mention the Lagrange spectrum for matrices and an easy result that can be deduced. For more details on this spectrum see the recent preprint \cite{Kleinbock}.
\subsection{Lagrange spectrum for matrices}
One can readily define a multidimensional version of the Lagrange spectrum. That is, given norms $\|\cdot\|_{1}$ on $\R^{n}$ and $\|\cdot\|_{2}$ on $\R^{m}$ and a matrix $\bX\in \R^{m\times n}$ define
\begin{equation*}
    \Theta_{m,n}(\bX, \|\cdot\|_{1}, \|\cdot\|_{2}):=\underset{\bq \in \Z^{n}}{\liminf_{\bq \to \infty}}\|\bq\|_{1}^{\frac{n}{m}} \min_{\bp\in \Z^{m}}\|\bX \bq - \bp\|_{2}.
\end{equation*}
Again, it is well known via Minkowski Convex body theorem that there exists $D_{m,n}$ dependent on norms $\|\cdot\|_{1}, \|\cdot\|_{2}$ such that $\Theta_{m,n}(\bX, \|\cdot\|_{1}, \|\cdot\|_{2})\leq D_{m,n}$ for all $\bX \in \R^{m\times n}$. Define
\begin{equation*}
    \cL_{m,n}(\|\cdot\|_{1}, \|\cdot\|_{2}):=\left\{ \Theta_{m,n}(\bX, \|\cdot\|_{1}, \|\cdot\|_{2}) : \bX \in \R^{m\times n}\right\}\subseteq [0,D_{m,n}].
\end{equation*}
For $n=1$ we return to the simultaneous version. Very little is known about the properties of these sets other than the recent work of Kleinbock \cite{Kleinbock} showing these sets are dense on the intervals $[0,\sup_{\bX\in \R^{m\times n}} \Theta_{m,n}(\bX, \|\cdot\|_{1}, \|\cdot\|_{2})]$. For general pairs $(m,n)$ it remains unknown whether these sets are uncountable. We give the following elementary observation, that does not appear to have been previously recorded.
\begin{theorem}\label{thm: matrices case}
   We have that
    \begin{equation*}
        \cL_{n,n}(\|\cdot\|_{\infty}, \|\cdot\|_{\infty})\supseteq \cL_{1},
    \end{equation*}
    and so $\cL_{n,n}(\|\cdot\|_{\infty}, \|\cdot\|_{\infty})$ contains a Hall's ray. Moreover,
    \begin{equation*}
        \liminf_{\varepsilon\to 0}\dimh \left\{\bX\in \R^{n\times n}: \Theta_{n,n}(\bX, \|\cdot\|_{\infty}, \|\cdot\|_{\infty})=\varepsilon \right\}\geq n^2-2n+2 \, .
    \end{equation*}
\end{theorem}
\begin{proof}
For ease of notation let $\Theta_{n,n}(\bX,\|\cdot\|_{\infty}, \|\cdot\|_{\infty})=\Theta_{n,n}(\bx)$. This can be seen by noting that for matrices $\bX_{i}\in \R^{n_{i}\times n_{i}}$ with $\sum_{i=1}^{k}n_{i}=n$ the matrix 
\begin{equation*}
    \bX=diag(\bX_{1}, \ldots, \bX_{k})
\end{equation*}
satisfies
\begin{equation}\label{eq: splitting matrices}
    \Theta_{n,n}(\bX)=\min_{1\leq i \leq k}\Theta_{n_{i},n_{i}}(\bX_{i}).
\end{equation}
To see why, suppose that $\min_{1\leq i \leq k}\Theta_{n_{i},n_{i}}(\bX_{i})=C_{1}$.
Then, for any $\eta>0$ and all $Q>Q_{\eta}$ vectors $\bq_{i}\in \Z^{n_{i}}$ with $\|\bq\|_{\infty}\geq Q$ we have that
\begin{equation*}
    |\bq_{i}\|_{\infty}\min_{\bp_{i}\in \Z^{n_{i}}}\|\bX_{i} \bq_{i}-\bp_{i}\|_{\infty}\geq C_{1}-\eta.
\end{equation*}
Given any $\bq=(\bq_{1},\ldots, \bq_{k})\in \Z^{n}$ choose $1\leq j \leq k$ such that $\|\bq\|_{\infty}=\|\bq_{j}\|_{\infty}$. Then 
\begin{align*}
    \|\bq\|_{\infty}\min_{\bp\in \Z^{n}}\|\bX \bq-\bp\|_{\infty}&=\|\bq\|_{\infty}\max_{1\leq i \leq k}\min_{\bp_{i}\in \Z^{n_{i}}}\|\bX_{i}\bq_{i}-\bp_{i}\|\\
    &\geq  \|\bq_{j}\|_{\infty}\min_{\bp_{j}\in \Z^{n_{j}}}\|\bX_{j}\bq_{j}-\bp_{j}\|_{\infty}\geq C_{1}-\eta.
\end{align*}
Since this was true for any $\eta>0$ we obtain the lower bound of \eqref{eq: splitting matrices}. For the corresponding upper bound suppose without loss of generality that $\min_{1\leq i \leq k}\Theta_{n_{i},n_{i}}(\bX_{i})=\Theta_{n_{1},n_{1}}(\bX_{1})= C_1$. Let the sequence $(\bp_{i,1},\bq_{i,1})_{i \in \N}$ in $\Z^{n_{1}}\times \Z^{n_{1}}$ be a sequence attaining
\begin{equation*}
    \lim_{i\to\infty}\left(\|\bq_{i,1}\|_{\infty}\|\bX_{1}\bq_{i,1}-\bp_{i,1}\|_{\infty}\right)=C_{1}.
\end{equation*}
Then taking the sequence of pairs
\begin{equation*}
    (\bp_{i},\bq_{i})=((\bp_{i,1}, 0, \ldots, 0), (\bq_{i,1},0,\ldots, 0))\in \Z^{n}\times \Z^{n}
\end{equation*}
one readily sees that $\lim_{i \to \infty}\|\bq_{i}\|_{\infty}\|\bX \bq_{i}-\bp_{i}\|_{\infty}=C_{1}$, thus $\Theta_{n,n}(\bX)\leq C_{1}$ obtaining our corresponding upper bound. To see the first part of the Theorem one can take $\bX=diag(x,\ldots, x)$ for some $x\in \R$ satisfying $\Theta_{1}(x)=c\in \cL_{1}$. Then \eqref{eq: splitting matrices} immediately gives $\Theta_{n,n}(\bX)=c\in \cL_{1}$. For the second part of the theorem, we have by \cite{Simmons2018}, that
\begin{equation*}
    \dimh \{\bX\in \R^{n\times n}: \Theta_{n,n}(\bX)>\varepsilon\}=n^{2}-V_{n,\|\cdot\|_{\infty}} \varepsilon^{n}+o(\varepsilon^{n}),
\end{equation*}
and by Moreira \cite{Moreira2018} that $\lim_{\varepsilon\to 0}\dimh \{x \in \R : \Theta_{1}(x)=\varepsilon\}=1$. Take the set
\begin{equation*}
    \Lambda_{n}(\varepsilon):=\left\{diag(x, \bX_{1})\in \R^{n\times n}: \begin{array}{c}x \in \{x \in \R: \Theta_{1}(x)=\varepsilon\} \\ \bX_{1}\in \{\bX\in \R^{(n-1)\times (n-1)}: \Theta_{n-1,n-1}(\bX)>\varepsilon\}
    \end{array}\right\}.
\end{equation*}
Then for any $\bX\in \Lambda(\varepsilon)$ we have $\Theta_{n,n}(\bX)=\varepsilon$. Since the block diagonal map $(x, \bX_{1})\mapsto diag(x, \bX_{1})$ is bi-Lipschitz and by a well-known result of Cartesian products of fractal sets \cite[\S 7]{Falconer2013} we have that
\begin{align*}
    \liminf_{\varepsilon \to 0} \dimh \Lambda_{n}(\varepsilon) &\geq \lim_{\varepsilon\to 0}\Bigg( \dimh \{x \in \R : \Theta_{1}(x)=\varepsilon\} \\
      & \qquad + \dimh \left\{\bX\in \R^{(n-1)\times (n-1)}: \Theta_{n-1,n-1}(\bX)>\varepsilon\right\} \Bigg)\\
    &\geq 1 +(n-1)^{2}
\end{align*}
%completing the proof.
\end{proof}
As a final remark we have the following simple consequence of Theorem~\ref{main}.

\begin{corollary} \label{cor: levelsets or uncountable}
    At least one of the following is true:
    \begin{enumerate}[i)]
    \item $\cL_{d}(\|\cdot\|_{2})$ is uncountable.
    \item
    
        $\underset{\varepsilon\to 0}{\limsup} \dimh \left\{ \bx \in \R^{d}: \Theta_{d}(\bx,\|\cdot\|_{2})=\varepsilon \right\}=d$.
    
    \end{enumerate}
\end{corollary}

\begin{proof}
    Suppose that $\cL_{d}(\|\cdot\|_{2})$ is countable, otherwise we are done. Take the sequence $\varepsilon_{i}=2^{-i}$ and function $\Delta(\varepsilon)=\log_2 (\varepsilon^{-1})$ and let $I_{i}:=[\varepsilon_{i}, \varepsilon_{i}+i\varepsilon_{i}^{d+1}]$, for $i\in \N$. By Theorem~\ref{main}
    \begin{equation*}
        \dimh \left\{\bx\in \R^{d}: \Theta_{d}(\bx,\|\cdot\|_{2})\in I_{i} \right\} \to d \quad \text{ as} \quad i\to\infty.
    \end{equation*}
    Using the countable stability of Hausdorff dimension, and our assumption that $\cL_{d}(\|\cdot\|_{2})$ is countable, we have that
    \begin{equation*}
        \dimh \left\{\bx\in \R^{d}: \Theta_{d}(\bx,\|\cdot\|_{2})\in I_{i} \right\}=\sup_{\varepsilon \in \cL_{d}(\|\cdot\|_{2})\cap I_{i}}\dimh \left\{\bx\in \R^{d}: \Theta_{d}(\bx,\|\cdot\|_{2})=\varepsilon \right\}.
    \end{equation*}
    In particular, there exists $\tilde{\varepsilon}_{i}\in \cL_{d}(\|\cdot\|_{2})\cap I_{i}$ such that
    \begin{equation*}
        \dimh \left\{\bx\in \R^{d}: \Theta_{d}(\bx,\|\cdot\|_{2})=\tilde{\varepsilon}_{i} \right\}>\dimh \left\{\bx\in \R^{d}: \Theta_{d}(\bx,\|\cdot\|_{2})\in I_{i} \right\}-2^{-i},
    \end{equation*}
    and so, by Theorem~\ref{main} we have
    \begin{equation*}
        \lim_{i\to\infty} \dimh \left\{\bx\in \R^{d}: \Theta_{d}(\bx,\|\cdot\|_{2})=\tilde{\varepsilon}_{i} \right\}=d.
    \end{equation*}
\end{proof}
We remark here that this could similarly be proven via \cite{Simmons2018}. The key ingredient is that the window has positive dimension. 

\end{document}